\documentclass[11pt]{amsart}

\usepackage{epigamath}

\usepackage[english]{babel}

\usepackage{amscd, amssymb, amsmath, mathrsfs, amsthm}
\usepackage[utf8]{inputenc}
\usepackage{tikz-cd}
\usetikzlibrary{decorations.pathreplacing}
\usepackage{booktabs}
\usepackage{color}

\usepackage{multicol}

\usepackage{microtype}

\usepackage{xy}
\usepackage{graphicx}
\usepackage{fancyhdr} % used by RCS
\usepackage{enumerate}
\newtheorem{prop}{Proposition}[section]
\newtheorem{theo}[prop]{Theorem}

\newtheorem{lem}[prop]{Lemma}
\newtheorem{coro}[prop]{Corollary}

\theoremstyle{definition}
\newtheorem{defi}[prop]{Definition}

\theoremstyle{remark}
\newtheorem{rmk}[prop]{Remark}
\newtheorem{exe}[prop]{Example}

\newtheoremstyle{monThm}
{} %Espace avant
{} %Espace après
{\itshape} %police du corps du texte
{} %Indentation
{\bf\scshape} %Police du titre
{.} %Ponctuation après le titre
{\newline} % espace après le titre
{} % Thm head spec

\theoremstyle{monThm}
\newtheorem{MonThm}{Theorem}

\newcommand{\C}{\mathbb{C}}

\newcommand{\EE}{\mathcal{E}}

\newcommand{\FF}{\mathcal{F}}

\renewcommand{\H}{\mathbb{H}}
\newcommand{\Hh}{\mathrm{H}}
\newcommand{\HH}{\mathcal{H}}
\newcommand{\h}{\textbf{h}}
\newcommand{\K}{\mathrm{K}}
\newcommand{\KK}{\mathcal{K}}

\newcommand{\LL}{\mathcal{L}}

\newcommand{\II}{\mathcal{I}}
\newcommand{\MM}{\mathcal{M}}
\newcommand{\N}{\mathbb{N}}
\newcommand{\NN}{\mathcal{N}}

\newcommand{\OO}{\mathcal{O}}
\renewcommand{\P}{\mathbb{P}}

\newcommand{\q}{\textbf{q}}
\newcommand{\R}{\mathbb{R}}
\newcommand{\RR}{\mathcal{R}}

\newcommand{\U}{\mathrm{U}}
\newcommand{\UU}{\mathcal{U}}

\newcommand{\VV}{\mathcal{V}}
\newcommand{\WW}{\mathcal{W}}

\newcommand{\Y}{\mathcal{Y}}
\newcommand{\Z}{\mathbb{Z}}

\renewcommand{\epsilon}{\varepsilon}
\renewcommand{\phi}{\varphi}

\newcommand{\equaldef}{\overset{\textrm{def}}{=}}

\newcommand{\norm}[1]{\Vert #1 \Vert}

\newcommand{\SO}{\mathrm{SO}}

\newcommand{\SL}{\mathrm{SL}}
\newcommand{\GL}{\mathrm{GL}}
\newcommand{\PSL}{\mathrm{PSL}}

\newcommand{\SU}{\mathrm{SU}}
\newcommand{\PU}{\mathrm{PU}}

\newcommand{\Sp}{\mathrm{Sp}}

\newcommand{\Spin}{\mathrm{Spin}}

\newcommand{\ProjC}[1]{\mathbb{C}\mathbf{P}^{#1}}

\newcommand{\Gr}{\text{Gr}}

\newcommand{\sslash}{\mathbin{/\mkern-6mu/}}

\newcommand{\End}{\mathrm{End}}
\newcommand{\Isom}{\mathrm{Isom}}

\newcommand{\Hom}{\mathrm{Hom}}
\newcommand{\Rep}{\mathfrak{X}}

\newcommand{\Id}{\mathrm{Id}}
\newcommand{\I}{\mathrm{I}}

\newcommand{\tr}{\mathrm{tr}}

\newcommand{\function}[5]{
#1 :\left\{\begin{array}{rcl}
#2 & \longrightarrow & #3 \\
#4 & \longmapsto & #5
\end{array}\right.
}

\newcommand{\dev}{\mathrm{dev}}

\newcommand{\rank}{\text{rk}}
\renewcommand{\hat}[1]{\widehat{#1}}

\renewcommand{\sslash}{\slash \! \slash}

\EpigaVolumeYear{5}{2021}
\EpigaArticleNr{6}
\ReceivedOn{November 5, 2019}
\InFinalFormOn{January 25, 2021}
\AcceptedOn{February 9, 2021}

\title{Compact connected components in relative character varieties of punctured spheres}
\titlemark{Compact components in relative character varieties}

\author{Nicolas Tholozan}
\address{Département de Mathématiques et Applications, UMR 8553,
\'Ecole Normale Sup\'erieure -- PSL Université, 
75005 Paris, 
France}
\email{nicolas.tholozan@ens.fr}

\author{Jérémy Toulisse}
\address{Laboratoire de Math\'ematiques J.A. Dieudonn\'e, 
UMR 7351 CNRS UNS, Universit\'e Côte d'Azur, 
06108 Nice Cedex 02, 
France} 
\email{jtoulisse@unice.fr}

\authormark{N. Tholozan and J. Toulisse}

\AbstractInEnglish{\footnotesize{We prove that some relative character varieties of the fundamental group of a punctured sphere into the Hermitian Lie groups $\textrm{SU}(p,q)$ admit compact connected components. The representations in these components have several counter-intuitive properties. For instance, the image of any simple closed curve is an elliptic element. These results extend a recent work of Deroin and the first author, which treated the case of $\textrm{PU}(1,1) = \textrm{PSL}(2,\R)$.

Our proof relies on the non-Abelian Hodge correspondence between relative character varieties and parabolic Higgs bundles. The examples we construct admit a rather explicit description as projective varieties obtained via Geometric Invariant Theory.}}



\TitleInFrench{Composantes compactes de variétés de caractères relatives de sphères épointées}

\AbstractInFrench{\scriptsize{Nous prouvons que certaines variétés de caractères relatives du groupe fondamental d'une sphère épointée dans les groupes de Lie hermitiens $\textrm{SU}(p,q)$ admettent des composantes connexes compactes. Les représentations dans ces composantes ont plusieurs propriétés contre-intuitives. Par exemple, l'image de toute courbe fermée simple est un élément elliptique. Ces résultats étendent un travail récent de Deroin et du premier auteur, qui traitait le cas de $\textrm{PU}(1,1) = \textrm{PSL}(2,\R)$.

Notre preuve repose sur la correspondance en théorie de Hodge non abélienne entre les variétés de caractères relatives et les fibrés de Higgs paraboliques. Les exemples construits se décrivent de façon assez explicite comme  des variétés projectives obtenues via la théorie géométrique des invariants.}}

\acknowledgement{The authors acknowledge support from U.S. National Science Foundation grants DMS 1107452, 1107263, 1107367 "RNMS: GEometric structures And Representation varieties" (the GEAR Network). Nicolas Tholozan's research is partially supported by the Agence Nationale de la Recherche through the grant DynGeo
(ANR-11-BS01-013).}

\begin{document}

%%%%%%%%%%%%%%%%%%%%%%%%%%%%%%%
% Title page
%%%%%%%%%%%%%%%%%%%%%%%%%%%%%%%

\removeabove{0.6cm}
\removebetween{0.6cm}
\removebelow{0.6cm}

\maketitle

\begin{prelims}

\DisplayAbstractInEnglish

\bigskip

\DisplayKeyWords

%\smallskip

\DisplayMSCclass

\bigskip

\languagesection{Fran\c{c}ais}

\bigskip

\DisplayTitleInFrench

\medskip

\DisplayAbstractInFrench

\end{prelims}

%%%%%%%%%%%%%%%%%%%%%
% Table of Contents
%%%%%%%%%%%%%%%%%%%%%

%\newpage

\setcounter{tocdepth}{1} 

\tableofcontents

%%%%%%%%%%%%%%%%%%%%%
% Content begins here
%%%%%%%%%%%%%%%%%%%%%

\section{Introduction}

Let $\Sigma_{g,s}$ be an oriented surface of genus $g$ with $s$ punctures, let $\Gamma_{g,s}$ denote its fundamental group and let $G$ be a non compact semi-simple Lie group. The character variety $\Rep(\Sigma_{g,s}, G)$ is roughly the space of morphisms from $\Gamma_{g,s}$ to $G$ up to conjugation by $G$. Character varieties of closed surfaces (\emph{i.e.} without puncture) carry a natural symplectic structure (constructed by Atiyah-Bott \cite{AtiyahBott} and Goldman \cite{GoldmanSymplectic}), which is preserved by the action of the mapping class group of the surface. The study of these character varieties, from a topological, algebraic, geometric or dynamical point of view, has attracted a lot of interest in the last decades.

The fundamental group of a surface with punctures is a free group, so its character varieties carry an action of the outer automorphism group of this free group. These character varieties, however, do not reflect all the geometric properties of the surface. Indeed, two punctured surfaces with the same Euler characteristic (such as $\Sigma_{0,3}$ and $\Sigma_{1,1}$) have the same fundamental group, hence the same character varieties.

One can nevertheless endow these character varieties with some extra structure that truly depends on the topology of $\Sigma_{g,s}$. Indeed, Goldman's construction has been extended to define a Poisson structure on $\Rep(\Sigma_{g,s}, G)$ (\cite{GuruprasadETAl,Lawton}), which is preserved by the action of the mapping class group of $\Sigma_{g,s}$. On the smooth locus, this Poisson structure defines a foliation by symplectic leaves. These leaves are the so called \emph{relative character varieties}, and consist of representations of $\Gamma_{g,s}$ into $G$ for which the conjugacy classes of the images of the loops going around the punctures are fixed. Given an element $h\in G^s$, we will denote by $\Rep_h(\Sigma_{g,s},G)$ the relative character variety whose representations map the loop around the $i$\textsuperscript{th} puncture to the conjugacy class of $h_i$ (here $i\in \{1,...,s\}$). In many aspects, these relative character varieties appear as the analogs for punctured surfaces of character varieties.

\vskip\baselineskip
Deroin and the first author studied in \cite{DeroinTholozan} the relative character varieties of the punctured sphere $\Sigma_{0,s}$ into $\PSL(2,\R) \cong \textrm{PSU}(1,1)$. For all $s\geq 3$, they exhibit representations with surprising properties that they call \emph{supra-maximal representations}, and prove that these representations form compact connected components of certain relative character varieties. Prior to their work, such compact components were only known to exist for $s=4$ \cite{BenedettoGoldman}.

The existence of such compact components may come out as a surprise, given the great flexibility of surface group representations. Indeed, Deroin--Tholozan's supra-maximal representations have several counter-intuitive properties:
\begin{itemize}
\item Though these representations are typically Zariski dense, the image of any \emph{simple} closed curve has a fixed point in the hyperbolic plane $\H^2$.

\item More generally, the eigenvalues of the image of a closed curve $c$ are bounded by a constant that depends only on the number of self-intersections of $c$.

\item For any such representation $\rho$ and any identification of $\Sigma_{0,s}$ with a punctured Riemann sphere, there exists a $\rho$-equivariant holomorphic map from $\widetilde{\Sigma}_{0,s}$ to $\H^2$.
\end{itemize}
The goal of the present paper is to construct such components for a larger class of Lie groups.

\vskip\baselineskip
To this end, we will exploit a very powerful tool for understanding the topology of relative character varieties: the \emph{non-Abelian Hodge correspondence}. When we fix a Riemann surface structure on the closed surface $\Sigma_{g,0}$, this correspondence associates to every irreducible linear representation of $\Gamma_{g,0}$ a \emph{stable Higgs bundle} over the Riemann surface. One can use tools from complex geometry to understand topological properties of the moduli space of Higgs bundles, and infer results on character varieties (see \cite{HitchinTeichComp}). The non-Abelian Hodge correspondence has been extended by Simpson and others \cite{SimpsonNonCompact,BiquardGarcia-Prada} to a correspondence between representations of punctured surfaces and \emph{parabolic Higgs bundles}.

Most of the theorems (if not all) about the topology of (relative) character varieties of surfaces can be recovered and improved using non-Abelian Hodge theory. As a matter of fact, shortly after Deroin--Tholozan's results, Mondello used non-Abelian Hodge theory to give a complete description of relative character varieties of punctured surfaces into $\PSL(2,\R)$ \cite{Mondello}, showing in particular that Deroin--Tholozan's compact components could be recovered via Higgs bundle techniques.

\vskip\baselineskip 
Here, we use non-Abelian Hodge theory to construct compact components in relative character varieties of punctured spheres into Lie groups of Hermitian type. We will prove the following theorem:
 
 \begin{MonThm} \label{t:ExistenceCompactComponents}
 Let $G$ be one of the classical Hermitian Lie groups $\SU(p,q)$, $\Sp(2n,\R)$ and $\SO^*(2n)$. For any $s\geq 3$, there exists a tuple $h = (h_1, \ldots , h_s) \in G^s$ such that the relative character variety
 \[\Rep_h(\Sigma_{0,s},G)\]
 has a compact connected component which contains a Zariski dense representation.
 \end{MonThm}
 
 \begin{rmk}
Most of the paper focuses on representations into $\SU(p,q)$. We discuss $\Sp(2n,\R)$ and $\SO^*(2n)$ in Section \ref{ss:MoreLieGroups}. Note that the only family of classical Hermitian Lie groups which our theorem does not cover is $\SO(2,n)$. See Section \ref{sss:SO(2,n)}.
 \end{rmk}
 
The starting point of the proof is to convert this theorem into a result about moduli spaces of parabolic $\SU(p,q)$-Higgs bundles. The general theory of Higgs bundles corresponding to representations into real Lie groups started with Hitchin's paper \cite{HitchinHiggsBundles}, and was fully sorted out by Garcia-Prada and Mundet i Riera \cite{GarciaPradaMundet}, while the parabolic version of the theory was developped by Biquard, Garcia-Prada and Mundet i Riera \cite{BiquardGarcia-Prada}. Parabolic $\SU(p,q)$-Higgs bundles have been specifically studied in \cite{Garcia-PradaLogares}, where the authors chose to put aside the punctured sphere case. 

Let us identify $\Sigma_{0,s}$ with the punctured Riemann sphere $\ProjC{1}\backslash \{x_1, \ldots, x_s\}$. Very briefly, parabolic Higgs bundles on the punctured Riemann sphere corresponding to representations into $\SU(p,q)$ are the data of $(\UU_\bullet, \VV_\bullet, \gamma, \delta)$ where $\UU_\bullet$ and $\VV_\bullet$ are holomorphic vector bundles over $\ProjC{1}$ of respective rank $p$ and $q$ with a \emph{parabolic structure} at each puncture (which involves a choice of \emph{weights} between $0$ and $1$) and $\gamma$ and $\delta$ are meromorphic sections of $\Hom(\UU,\VV)\otimes \KK$ and $\Hom(\VV,\UU)\otimes \KK$ respectively, with at most simple poles at the punctures and satisfying certain constraints depending on the weights (here $\KK$ is the canonical bundle on $\ProjC{1}$). 

One associates to representations into Hermitian Lie groups a topological invariant called the \emph{Toledo invariant} (see Paragraph \ref{ss:ToledoInvariant}), which is of main importance in the topology of character varieties. Under some open condition on the weights, we show that, if the Toledo invariant of the corresponding representation is small, then $\delta$ must vanish for stable Higgs bundles. Though this condition is very restrictive, we prove that it can be satisfied for $s > 2+ \frac{p}{q} + \frac{q}{p}$, giving rise to compact components in relative character varieties. The underlying holomorphic bundles $\mathcal{U}$ and $\mathcal{V}$ in these components are both direct sums of a fixed line bundle. This fact permits a parametrization of these components by certain quiver varieties called \emph{Kronecker varieties}. For more generic weights, the parabolic structure of the Higgs bundles give rise to a decoration of those quivers that we call \textit{feathered Kronecker varieties} and construct using Geometric Invariant Theory.

This construction only works for $s$ large enough. However, restricting these representations to subsurfaces, one can prove the existence of compact components for all $s$ (see Section \ref{ss:RestrictionSubsurface}).

\vskip\baselineskip
Finally, let us point out that the representations we obtain have the same kind of properties as supra-maximal representations into $\PSL(2,\R)$.

\begin{MonThm}\label{t-thm2}
Let $G$ be one of the classical Hermitian Lie groups $\SU(p,q)$, $\Sp(2n,\R)$ and $\SO^*(2n)$. There exists an open set in $\Hom(\Gamma_{0,s}, G)$ consisting of representations $\rho$ satisfying the following properties:
\begin{itemize}
\item For every homotopy class of simple closed curve $c$ in $\Sigma_{0,s}$,
the complex eigenvalues of $\rho(c)$ have modulus $1$.
\item More generally, for every $k\geq 0$, there is a constant $C(k)$ such that, if $c$ is the homotopy class of a closed curve with at most $k$ self interserctions, then the eigenvalues of $\rho(c)$ have modulus less than $C(k)$.
\item For every identification of $\Sigma_{0,s}$ with a punctured Riemann sphere, there is a $\rho$-equivariant holomorphic map from $\widetilde{\Sigma}_{0,s}$ to the symmetric space of $G$.
\end{itemize}
\end{MonThm}

\subsection*{Structure of the article}

Section \ref{s:relativecharactervariety} introduces some background notations and recalls briefly the non-Abelian Hodge correspondence between relative character varieties and parabolic Higgs bundles. In Section \ref{s:GIT}, we recall some classical facts about Geometric Invariant Theory and its application to the construction of certain  quiver varieties. Section \ref{s:mainsection} contains the core of our work. We give a condition for a component of the moduli space of parabolic $\SU(p,q)$-Higgs bundles to be compact and prove that, for some specific choices of weights, this component is indeed non empty and isomorphic to a Kronecker variety. In Section \ref{s:PropertiesOfOurReps}, we describe some properties of the representations in these components, proving  Theorems \ref{t:ExistenceCompactComponents} and \ref{t-thm2} for $\SU(p,q)$, under the condition $s>\frac{p}{q}+\frac{q}{p}$. Finally, in Section \ref{s:MoreComponents}, we discuss some extensions of our construction (to other Hermitian Lie groups, and to spheres with fewer punctures), concluding the proof of Theorems \ref{t:ExistenceCompactComponents} and \ref{t-thm2}.

\subsection*{Acknowledgments}

This work is very much indebted to Bertrand Deroin. Our motiviation for understanding compact components for relative character varieties into general Lie groups stemmed out of long discussions with him. We also thank Olivier Biquard for discussions from which the idea emerged that one could understand compact components via parabolic Higgs bundles, and Marco Maculan for helping us clarify some details of Geometric Invariant Theory. We also thank Julien March\'e and Maxime Wolff for valuable help concerning terminology, and the anonymous referees for helping us improve the exposition of the paper. Finally we are very thankful to the GEAR Network for giving us the opportunity to organize the workshop \textit{Relative character varieties and parabolic Higgs bundles}, where we learnt the details of the theory.

\section{Character varieties and Higgs bundles}
\label{s:relativecharactervariety}

\subsection{Relative character varietes}

Throughout the paper, $p$ and $q$ denote two positive integers. Let $\C^{p,q}$ be the complex vector space $\C^{p+q}$ equipped with the following Hermitian form of signature $(p,q)$:
\[\textbf{h}(z,w)= z_1\overline{w}_1+\cdots+z_p\overline{w}_p -\cdots-z_{p+q}\overline{w}_{p+q}.\]
We define $\U(p,q)$ as the subgroup of $\GL(p+q,\C)$ preserving $\textbf{h}$ and by $G=\SU(p,q)$ the subgroup of $\U(p,q)$ consisting of matrices of (complex) determinant $1$.

\smallskip
Given an element $g$ in a semi-simple Lie group $G$, we define $C(g)$ as the set of elements $h\in G$ such that the closures of the conjugacy orbits of $g$ and $h$ intersect. For instance, if $g$ is diagonalizable with distinct eigenvalues, then $C(g)$ is exactly the conjugacy orbit of $g$. In general, $C(g)$ consists in a finite union of conjugacy orbits, a single one of which is closed.

Let $\Sigma_{g,s}$ denote the oriented surface of genus $g$ with $s$ punctures. We denote by $\Gamma_{g,s}$ its fundamental group. Let $c_1, \ldots, c_s$ denote homotopy classes of loops going counter-clockwise around the punctures of $\Sigma_{g,s}$. Finally, fix $h = (h_1,\ldots, h_s) \in G^s$.

\begin{defi}\label{d:relativecharactervariety}
A morphism $\rho: \Gamma_{g,s} \to G$ is \emph{of type $h$} if $\rho(c_k)$ belongs to $C(h_k)$ for all $k\in\{1,\ldots,s\}$. The \textit{relative character variety} $\Rep_h(\Sigma_{g,s}, G)$ is defined as
\[\Rep_h(\Sigma_{g,s}, G) = \{ \rho : \Gamma_{g,s} \to \SU(p,q) \mid \rho(c_k) \in C(h_k)\}\sslash G~,\]
where $\sslash G$ denotes the equivalence relation identifying two representations $\rho$ and $\rho'$ when the closures of their conjugation orbits under $G$ intersect.
\end{defi}

\begin{rmk}
If $G$ is a complex algebraic group, then the quotient $\sslash G$ is an algebraic quotient in the sense of Geometric Invariant Theory (GIT). For real Lie groups, however, some subtleties make it harder to define a nice algebraic quotient.
\end{rmk}

\begin{rmk}
Most morphisms from $\Gamma$ to $G$ are irreducible and thus have a closed orbit under conjugation. Actually, if $s>0$, then all morphisms of type $h$ are irreducible for a generic choice of $h$, so that $\Rep_h(\Sigma_{g,s}, G)$ is simply the space of conjugacy classes of morphisms from $\Gamma_{g,s}$ to $G$ of type $h$.
\end{rmk}

\subsubsection{Toledo invariant}\label{ss:ToledoInvariant}

The Lie group $\SU(p,q)$ is Hermitian, meaning that its symmetric space carries a $\SU(p,q)$-invariant K\"ahler structure. The K\"ahler form induces a continuous cohomology class $\omega$ of degree $2$ on $\SU(p,q)$ (seen as a topological group). Given a morphism $\rho: \Gamma_{g,0} \to \SU(p,q)$, one can pull back $\omega$ by $\rho$ to obtain an element of $H^2(\Gamma_{g,0}, \R) \cong H^2(\Sigma_{g,0}, \R) \cong \R$. For a surface with punctures, this construction does not work directly since $H^2(\Sigma_{g,s}, \R) =\{0\}$, but one can fix this issue using bounded cohomology \cite{BurgerBounded}. This eventually defines a function
\[\mathrm{Tol} : \Rep(\Gamma_{g,s}, G) \to \R\]
called the \emph{Toledo invariant}. The Toledo invariant is continuous and additive in the following sense:

\begin{prop}[see \cite{BurgerBounded}]
If $b$ is a simple closed curve on $\Sigma_{g,s}$ separating $\Sigma_{g,s}$ into two surfaces $\Sigma'$ and $\Sigma''$, then for every representation $\rho:\Gamma_{g,s} \to \SU(p,q)$,
\[\mathrm{Tol}(\rho) = \mathrm{Tol}(\rho_{|\pi_1(\Sigma')}) + \mathrm{Tol}(\rho_{|\pi_1(\Sigma'')|})~.\]
\end{prop}

For closed surfaces, the Toledo invariant takes values into $r\Z$ for some rational number $r$ (depending on the normalization of the cohomology class $\omega$). For punctured surfaces, the image of $\mathrm{Tol}$ is not discrete, but the congruence of $\mathrm{Tol}(\rho)$ modulo $r \Z$ only depends on $C(h_1), \ldots, C(h_s)$ \cite{BurgerBounded}. In particular, the Toledo invariant is constant on connected components of relative character varieties. We will denote by $\Rep_h^\tau(\Sigma_{g,s}, \SU(p,q))$ the set of representations $\rho \in \Rep_h(\Sigma_{g,s},\SU(p,q))$ such that $\mathrm{Tol}(\rho) = \tau$. We will sometimes call this space a \emph{relative component}, though it may not be connected in general.

\subsection{Parabolic Higgs bundles} The non-Abelian Hodge correspondence, developed by Hitchin, Simpson and  many others \cite{HitchinHiggsBundles,SimpsonHiggsBundles}, gives a homeomorphism between the character variety of a closed Riemann surface and the moduli space of Higgs bundles over this surface. This correspondence was refined by Simpson \cite{SimpsonNonCompact} and Biquard--Garc\'ia-Prada--Mundet i Riera \cite{BiquardGarcia-Prada} to a correspondence between relative character varieties of punctured Riemann surfaces and moduli spaces of \emph{parabolic Higgs bundles}. In this subsection, we introduce parabolic Higgs bundles and state the non-Abelian Hodge correspondence in the case that interests us.

\subsubsection{Parabolic vector bundles}

Let $X$ be a closed Riemann surface and $x$ a point on $X$. By a vector bundle over $X$, we always mean a holomorphic vector bundle. We denote vector bundles with curly letters and use the same capital letter to denote the corresponding sheaf of holomorphic sections.

\begin{defi}
Given a vector bundle $\EE$ over $X$, a \textit{parabolic structure at $x\in X$} is a filtration $\{E_\alpha\}_{\alpha\in \R}$ of the germs of meromorphic sections of $\EE$ at $x$, such that
\begin{itemize}
	\item $\{E_\alpha\}_{\alpha\in \R}$ is decreasing, that is $E_{\alpha'}\supset E_\alpha$ whenever $\alpha'\leq \alpha$.
	\item $\{E_\alpha\}_{\alpha\in \R}$ is left continuous, \emph{i.e.} for any $\alpha\in\R$, there exists $\epsilon>0$ with $E_{\alpha-\epsilon}=E_\alpha$.
	\item $E_0$ are the germs of holomorphic sections of $\EE$ at $x$ and $E_{\alpha+1}=E_\alpha(-x)$ (that is, $s\in E_\alpha$ if and only if $zs\in E_{\alpha+1}$, where $z$ is a local coordinate centered at $x$).
\end{itemize}
\end{defi}

\begin{rmk}
Given a vector bundle $\EE$ over $X$, one can always equip $\EE$ with the \textit{trivial parabolic structure} $\{E_\alpha \}_{\alpha\in\R}$ at $x\in X$ by setting $E_\alpha=E_0(-\lfloor\alpha\rfloor x)$, where $E_0$ is the space of germs of holomorphic sections of $\EE$ at $x$ and $\lfloor\alpha\rfloor$ is the integer part of $\alpha$.
\end{rmk}

Given two vector bundles $\EE$ and $\FF$ with a parabolic structure at $x\in X$, their direct sum and tensor product are respectively defined to be the bundles $\EE\oplus \FF$ and $\EE\otimes \FF$ with the parabolic structure $\left\{(E\oplus F)_\alpha \right\}_{\alpha\in\R}$ and $\{(E\otimes F)_\alpha \}_{\alpha\in\R}$ at $x$ where
\[\left\{\begin{array}{lll} \big(E \oplus F\big)_\alpha & = & E_\alpha \oplus F_\alpha \\
\big(E\otimes F \big)_\alpha & = & \mathrm{Span}\,\displaystyle{ \left(\bigcup_{\alpha_1+\alpha_2=\alpha} E_{\alpha_1}\otimes F_{\alpha_2}\right)~.}\end{array}\right. \]
If $\FF\subset \EE$ is a subbundle and $\EE$  has a parabolic structure at $x$, then $\FF$ inherits a parabolic structure $\{F_\alpha \}_{\alpha\in\R}$ where $F_\alpha$ are the germs of meromorphic section of $\FF$ at $x$ contained in $E_\alpha$.

Given a vector bundle $\EE$ with a parabolic structure at $x\in X$, we call a \textit{jump} of the filtration $\{E_\alpha\}_{\alpha\in\R}$  a number $\alpha\in\R$ such that $E_\alpha\neq E_{\alpha+\epsilon}$ for any $\epsilon>0$. The normalization condition $E_{\alpha+1}=E_\alpha(-x)$ implies that the parabolic structure is fully determined by the set $\{(E_{\alpha_1},\alpha_1),\ldots,(E_{\alpha_r},\alpha_r)\}$ where $\alpha_1,\ldots,\alpha_r\in[0,1)$ are the jumps of the filtration in $[0,1)$.

Denote by $\EE_x$ the fiber of $\EE$ over $x$. The evaluation map 
\[\textbf{ev}:\left\{\begin{array}{lll}  E_0 & \longrightarrow & \EE_x \\
s & \longmapsto & s(x) \end{array}\right.,\]
has kernel $E_1=E(-x)$. In particular, $\textbf{ev}$ gives an isomorphism $E_0/E_1\cong \EE_x$ and identifies $E_{\alpha_i}\subset E_0$ with a linear subspace $\EE_{x,i}\subset \EE_x$. In other words, a parabolic structure at $x$ is equivalent to a weighted flag 
\[\EE_x=\EE_{x,1}\supset \EE_{x,2}\supset\cdots\supset \EE_{x,r},~0\leq \alpha_1<\cdots<\alpha_r<1.\]

Setting $k_i=\dim\left(\EE_{x,i+1}/\EE_{x,i}\right)$, the \textit{parabolic type} of $\EE$ at $x$ is the $n$-tuple $(\tilde\alpha_1,\ldots,\tilde\alpha_n)\in [0,1)^n$ (where $n$ is the rank of $\EE$) such that $\alpha_1=\tilde\alpha_1=\cdots=\tilde\alpha_{k_1},~\alpha_2=\tilde\alpha_{k_1+1}=\cdots=\tilde\alpha_{k_1+k_2}$ and so on.

\begin{defi} A morphism $f: \EE \to \FF$ of vector bundles with parabolic structure of type $\alpha$ and $\beta$ at $x\in X$ is \textit{parabolic} if $\alpha_i\geq \beta_j$ implies $f(\EE_{x,i})\subset \FF_{x,j}$. It is \emph{strongly parabolic} if $\alpha_i\geq \beta_j$ implies
 $f(\EE_{x,i})\subset \FF_{x,j+1}$.
\end{defi}

Given a punctured Riemann surface $X$ (\emph{i.e.} a Riemann surface with cusps) with closure $\overline{X}=X\cup\{x_1,\ldots,x_s\}$, a \textit{parabolic bundle of rank $n$ over $X$} will be a rank $n$ vector bundle $\EE$ over $\overline{X}$ together with a parabolic structure at each $x_i\in\{x_1,\ldots,x_s\}$. We denote by $\EE_\bullet$ such a parabolic bundle. The \textit{parabolic type} of $\EE_\bullet$ is $\alpha=(\alpha^1,\ldots,\alpha^s)$ where $\alpha^i=(\alpha^i_1,\ldots,\alpha^i_n)\in[0,1)^n$ is the parabolic type of $\EE_\bullet$ at $x_i$. Finally, we denote by $\Hom(\EE_\bullet,\FF_\bullet)$ the bundle of strongly parabolic morphisms from $\EE_\bullet$ to $\FF_\bullet$, and write $\End(\EE_\bullet)=\Hom(\EE_\bullet,\EE_\bullet)$.

The \textit{dual} of a parabolic bundle $\EE_\bullet$ is the parabolic bundle $\EE_\bullet^\vee=\Hom(\EE_\bullet,\OO)$ where $\OO$ is the trivial bundle equipped with the trivial parabolic structure. Note that the underlying bundle of $\EE_\bullet^\vee$ is $\EE^\vee(-D)$ where $D$ is the effective divisor $x_1+\cdots+x_s$.

\begin{defi}
The \textit{parabolic degree} of a parabolic bundle $\EE_\bullet$ of type $\alpha$ over $X$ is the number
\[\deg(\EE_\bullet) \equaldef \deg(\EE) + \Vert \alpha \Vert,\]
where $\Vert \alpha\Vert= \sum_{i=1}^s\sum_{j=1}^n \alpha^i_j$.
\end{defi}

\begin{exe}[Parabolic line bundles over $\ProjC{1}$]
Recall that $\ProjC{1}$ carries a unique line bundle of degree $l$ for all $l\in \Z$, denoted $\OO(l)$. Let $X=\ProjC{1}\setminus\{x_1,\ldots,x_s\}$ be a punctured Riemann sphere.  A parabolic line bundle over $X$ is thus given by a parabolic structure over some $\OO(l)$ at each $x_i$. Since a flag in $\OO(l)_{x_i}$ is necessarily trivial, such a parabolic structure is simply given by its type $\alpha^i\in[0,1)$. We will denote this parabolic line bundle by $\OO\big(l+\sum_{i=1}^s \alpha^ix_i\big)$.

Consider now two parabolic line bundles $\LL_\bullet=\OO\big(l+\sum_{i=1}^s \alpha^ix_i\big)$ and $\MM_\bullet=\OO\big(m+\sum_{i=1}^s \beta^ix_i\big)$. Let $\epsilon^i=\lfloor \alpha^i+\beta^i\rfloor\in\{0,1\}$ be the integer part of $\alpha^i+\beta^i$ and set $\vert\epsilon\vert = \sum_{i=1}^s \epsilon^i$. The tensor product $\LL_\bullet\otimes \MM_\bullet$ is thus the line bundle of degree $l+m+\vert\epsilon\vert$ whose parabolic weight at $x_i$ is $\alpha^i+\beta^i - \epsilon^i\in [0,1)$.

If $\alpha^i< \beta^i$, then any morphism $f: \LL \to \MM$ of line bundles is strongly parabolic at $x_i$. On the other hand, if $\alpha^i\geq\beta^i$, a morphism $f: \LL \to \MM$ is strongly parabolic at $x_i$ only if it comes from a morphism $f: \LL \to \MM(-x_i)$. Therefore, $\Hom(\LL_\bullet,\MM_\bullet)=\OO(m-l-k)$, where $k$ is the number of punctures $x_i$ with $\alpha^i\geq\beta^i$.
\end{exe}

\subsubsection{Parabolic $\SU(p,q)$-Higgs bundles}\label{s:su(p,q)higgsbbundles}
Recall that $X=\overline{X}\setminus\{x_1,\ldots,x_s\}$. We denote by $D$ the effective divisor $\sum_{i=1}^s x_i$. Let $\mathcal{K}$ denote the canonical bundle of $X$. 

\begin{defi}
A \textit{parabolic Higgs bundle} over $X$ is a pair $(\EE_\bullet,\Phi)$ where $\EE_\bullet$ is a parabolic bundle and $\Phi$ is a holomorphic section of $\KK(D)\otimes \End(\EE_\bullet)$.
\end{defi}

\begin{rmk}
In this definition, the Higgs field $\Phi$ is a strongly parabolic morphism from $\EE_\bullet$ to $\EE_\bullet\otimes \KK(D)$. Stricto sensu, this is the definition of a \emph{strongly parabolic Higgs bundle}. In this paper, however, we limit ourselves to strongly parabolic Higgs bundles, so we omit this precision.
\end{rmk}

\begin{defi}
A \textrm{parabolic $\SU(p,q)$-Higgs bundle} is a parabolic Higgs bundle $(\EE_\bullet,\Phi)$ such that
\begin{itemize}
	\item[--] $\EE_\bullet=\UU_\bullet\oplus \VV_\bullet$ where $\UU_\bullet$ and $\VV_\bullet$ are parabolic bundles of respective rank $p$ and $q$.
	\item[--] With respect to the splitting $\EE_\bullet=\UU_\bullet\oplus \VV_\bullet$, the Higgs field $\Phi$ takes the form $\left(\begin{array}{ll} 0 & \delta \\ \gamma & 0 \end{array}\right)$ where $\gamma\in H^0\big(\KK(D)\otimes\Hom(\UU_\bullet,\VV_\bullet)\big)$ and $\delta\in H^0\big(\KK(D)\otimes\Hom(\VV_\bullet,\UU_\bullet) \big)$.
	\item[--] $\det(\EE_\bullet)=\det(\UU_\bullet)\otimes\det(\VV_\bullet)$ has trivial parabolic structure and is isomorphic to the trivial bundle $\OO$.
\end{itemize}
\end{defi}
We will denote parabolic $\SU(p,q)$-Higgs bundles by $(\UU_\bullet\oplus \VV_\bullet,\gamma\oplus\delta)$.

\begin{defi}
A parabolic $\SU(p,q)$-Higgs bundle $(\UU_\bullet\oplus \VV_\bullet,\gamma\oplus\delta)$ is \textit{stable} (respectively \emph{semistable}) if any $\gamma\oplus\delta$-invariant subbundle $\UU_\bullet'\oplus \VV_\bullet'$ where $\UU'_\bullet\subset \UU_\bullet$ and $\VV'_\bullet\subset \VV_\bullet$ satisfies $\deg(\UU'_\bullet\oplus \VV'_\bullet)<0$ (respectively $\deg(\UU'_\bullet\oplus \VV'_\bullet)\leq 0$). It is \textit{polystable} if it is the direct sum of stable parabolic Higgs bundles of degree~$0$.
\end{defi}

The \textit{parabolic type} of a parabolic $\SU(p,q)$-Higgs bundle $(\UU_\bullet\oplus \VV_\bullet,\gamma\oplus\delta)$ is the pair $(\alpha,\beta)$ where $\alpha$ and $\beta$ are respectively the parabolic types of $\UU_\bullet$ and~$\VV_\bullet$.

Note that the parabolic type of the parabolic line bundle $\det(\UU_\bullet)\otimes \det(\VV_\bullet)$ at $x_i$ is given by the quantity $\sum_{j=1}^p \alpha_j^i + \sum_{j=1}^q\beta_j^i$. The condition $\det(\UU_\bullet)\otimes\det(\VV_\bullet)=\mathcal O$ implies that $\sum_{j=1}^p \alpha_j^i + \sum_{j=1}^q\beta_j^i$ is integral. We denote by $\WW(s,p,q)$ the subset of $\left(\R^p\right)^s \times \left( \R^q\right)^s$ consisting of tuples $(\alpha,\beta)$ such that
\[0\leq \alpha_1^i \leq \cdots \leq \alpha_p^i <1~,\]
\[0\leq \beta_1^i \leq \cdots \leq \beta_q^i < 1\]
and
\[\sum_{j=1}^p \alpha_j^i+ \sum_{j=1}^q \beta_j^i\in \N\]
for all $i\in \{1,\ldots, s\}$. Elements of $\WW(s,p,q)$ will be called \emph{$\SU(p,q)$-multiweights}.

\subsubsection{Moduli spaces}

Moduli spaces of semistable Higgs bundles can be constructed using Geometric Invariant Theory. This was done by Yokogawa in a fairly general setting which includes parabolic Higgs bundles \cite{Yokogawa1,Yokogawa2}. In our setting, we have the following
\begin{theo}
For every $\SU(p,q)$-multiweight $(\alpha,\beta)$, there is a normal quasi-projective variety $\MM(\alpha,\beta)$ and, for every algebraic family of semistable parabolic $\SU(p,q)$-Higgs bundles of type $(\alpha,\beta)$ over a base $B$, an algebraic map $f_B: B\to \MM(\alpha,\beta)$ such that:
\begin{itemize}
\item Every point of $\MM(\alpha, \beta)$ is in the image of $f_B$ for some family $B$,
\item given two such families over bases $B$ and $B'$ and two points $x\in B$ and $y\in B'$, if the Higgs bundles attatched to these two points are polystable, then they are isomorphic if and only if $f_B(x) = f_{B'}(y)$.
\end{itemize}
\end{theo}

There is a continuous map
\[\function{\varphi}{\MM(\alpha,\beta)}{\Z}{(\UU_\bullet\oplus \VV_\bullet,\gamma\oplus\delta)}{\deg(\UU)-\deg(\VV)~.}\]
 In particular, $\MM(\alpha,\beta)= \bigsqcup_{d\in\Z} \MM(\alpha,\beta,d)$, where $\MM(\alpha,\beta,d)=\varphi^{-1}(d)$.

\begin{rmk}\label{r-degreeU}
Using the equality $\deg(\UU_\bullet)+\deg(\VV_\bullet)= \deg(\UU)+\deg(\VV)+\Vert\alpha\Vert+\Vert\beta\Vert=0$, one sees that if $(\UU_\bullet\oplus \VV_\bullet,\gamma\oplus\delta)\in\MM(\alpha,\beta,d)$, then 
\[\left\{\begin{array}{lll}
\deg(\UU) & = & -\frac{1}{2}(\Vert\alpha\Vert+\Vert\beta\Vert-d) \\
\deg(\VV) & = & -\frac{1}{2}(\Vert\alpha\Vert+\Vert\beta\Vert +d)~.\end{array}\right.\]
\end{rmk}

\begin{exe}\label{ex:p=q=1} In this example, we describe the case $p=q=1$ and $X=\ProjC{1}\setminus\{x_1,\ldots,x_s\}$. Recall that $\OO(d)$ denotes the only line bundle of degree $d$ over $\ProjC{1}$.

In a parabolic $\SU(1,1)$-Higgs bundle $(\UU_\bullet\oplus \VV_\bullet,\gamma\oplus\delta)$, $\UU_\bullet$ and $\VV_\bullet$ are parabolic line bundles $\OO(l+\sum_{i=1}^s\alpha^i x_i)$ and $\OO(m+\sum_{i=1}^s \beta^ix_i)$ respectively. The condition $\det(\UU_\bullet)\otimes\det(\VV_\bullet)=\mathcal O$ implies that $\alpha^i+\beta^i \in\{0,1\}$ for all $i=1,\ldots,s$. 

We now consider the case where $\alpha^i < \beta^i$ for all $i\in\{1,\ldots,s\}$. This implies that $\alpha^i=1-\beta^i\in (0,\frac{1}{2})$, and $\deg(\UU)+\deg(\VV)=-s$ so $d=\deg(\UU)-\deg(\VV)=2l+s$. One sees that if the parity of $d$ is different from the one of $s$, then $\MM(\alpha,\beta,d)=\emptyset$. Consider now the Higgs field. Since $\alpha^i<\beta^i$, we have
\[ \left\{ \begin{array}{lllll} 
\gamma\in H^0(\KK(D)\otimes\Hom(\UU_\bullet,\VV_\bullet)) & = & H^0(\KK(D)\otimes\Hom(\UU,\VV)) & = & H^0(\OO(-2l-2)) \\
\delta\in H^0(\KK(D)\otimes\Hom(\VV_\bullet,\UU_\bullet)) & = & H^0(\KK\otimes\Hom(\VV,\UU)) & = & H^0\big(\OO(2l+s-2))\end{array}\right..\]

Finally, let us describe a particular component. Suppose $s$ is odd and look at the case $d=1$. In particular, $l=\deg(\UU)=-\frac{s-1}{2}$ and $\deg(\VV)=-\frac{s+1}{2}$. It follows that 
\[\deg(\KK\otimes\Hom(\VV,\UU))=2l+s-2=-1<0\]
and so $\delta=0$. One also sees that $\gamma\in H^0(\OO(s-3))$. Such a parabolic Higgs bundle is stable if and only if $\norm{\beta} =\sum_{i=1}^s\beta^i<\frac{s+1}{2}$ and $\gamma\neq 0$.

Two such stable Higgs bundles given by $\gamma,\gamma'\in H^0(\OO(s-3))$ are isomorphic if and only if $[\gamma]=[\gamma']$ in $\P\left(H^0(\OO(s-3)) \right)$. We thus have $\MM(\alpha,\beta,1)\cong \ProjC{s-3}$.

This component corresponds to a lift to $\SU(1,1)$ of a supra-maximal component in $\PSL(2,\R)\cong \PU(1,1)$ studied by Deroin and the first author in \cite{DeroinTholozan}. The goal of the present paper is to extend this phenomenon to $\SU(p,q)$.
\end{exe}

\begin{defi}
Given a rank $n$ parabolic Higgs bundle $(\EE_\bullet,\Phi)$ over $X$, one can associate a vector $\Pi_{\textrm{Hitchin}}(\EE_\bullet,\Phi) \overset{\textrm{def}}{=} (\tr(\Phi^i))_{1\leq i \leq n}\in  H^0(\KK(D)^i )$. This defines the \textit{Hitchin fibration}
\[\Pi_{\textrm{Hitchin}}:  \MM(\alpha,\beta)  \longrightarrow  \bigoplus_{i=1}^{n} H^0(\KK(D)^i)~.\]
\end{defi}

The following was proven by Hitchin \cite{HitchinIntegrable} for closed Riemann surfaces and extended by Yokogawa \cite{Yokogawa2} to the parabolic case:

\begin{theo}\label{t:PropernessHitchinMap}
The Hitchin fibration is proper.
\end{theo}

\subsubsection{Non-abelian Hodge correspondence}\label{ss:NAH} Our main interest in the moduli space $\MM(\alpha,\beta)$ of parabolic $\SU(p,q)$-Higgs bundles of type $(\alpha,\beta)$ is the non-Abelian Hodge correspondence \cite{SimpsonNonCompact,BiquardGarcia-Prada}.

Let $(\alpha, \beta)$ be a $\SU(p,q)$-multiweight. Let $h=h(\alpha,\beta)$ be the tuple of matrices $(h_1,\ldots, h_s)$ where $h_k$ is the diagonal matrix of size $p+q$ with diagonal coefficients 
\[e^{2\pi i \alpha_1^k}, \ldots,e^{2\pi i \alpha_p^k}, e^{2\pi i \beta_1^k}, \ldots, e^{2\pi i \beta_q^k}~.\]
The definition of a $\SU(p,q)$-multiweight implies that each $h_k$ belongs to $\SU(p,q)$.

\begin{theo}[Non-Abelian Hodge correspondence]
There exist a homeomorphism $\mathsf{NAH}$ from the moduli space $\MM(\alpha,\beta)$ to the relative character variety $\Rep_h(\Sigma_{0,s},\SU(p,q))$. Through this correspondence, stable Higgs bundles correspond to irreducible representations.
\end{theo}

This correspondence is highly transcendental in the sense that it relies on the resolution of a geometric PDE on Hermitian metrics on the bundle. It is thus quite difficult to understand geometric properties of a representation associated to a given Higgs bundle. There are nevertheless a few attributes that transit well through the correspondence, among which is the Toledo invariant.

\begin{prop}
Let $(\UU_\bullet \oplus \VV_\bullet, \gamma \oplus \delta)$ be a polystable $\SU(p,q)$-Higgs bundle and let $\rho$ be its image by $\mathsf{NAH}$. Then (with the proper normalization of the Toledo invariant) one has
\[\mathrm{Tol}(\rho) = \deg(\VV_\bullet) - \deg(\UU_\bullet)~.\]
\end{prop}

\begin{rmk}
This Proposition is proved in \cite{Garcia-PradaLogares}, where a different normalization of the Toledo invariant is used. The one we use here is not standard, but it will simplify notations.
\end{rmk}

Recall that $\norm{\alpha}$ (resp. $\norm{\beta}$) denotes the sum of the weights $\alpha_j^i$ (resp. $\beta_j^i$) for all $j$ and $i$. By definition, we have
\[\deg(\VV_\bullet) - \deg(\UU_\bullet) = \deg(\VV) - \deg(\UU) + \norm{\beta} - \norm{\alpha}~.\]
It follows that the non Abelian Hodge correspondence restricts to a homeomorphism
\[\mathsf{NAH} : \MM(\alpha, \beta,d) \overset{\sim}{\longrightarrow} \Rep_h^{\norm{\beta} - \norm{\alpha}-d}\left(\Sigma_{0,s}, \SU(p,q)\right)~.\]

\subsubsection{Parabolic Higgs bundles over the Riemann sphere}\label{ss:degreevector} In this paper, we consider the special case of $X=\ProjC{1}\setminus\{x_1,\ldots,x_s\}$. Let us briefly recall the classification of holomorphic vector bundles over~$\ProjC{1}$. 

Recall that we denote by $\OO(d)$ the unique line bundle of degree $d$ over $\ProjC{1}$.

\begin{theo}[Birkhoff--Grothendieck]\label{t:Birkhoff-Grothendieck}
Any rank $n$ vector bundle $\EE$ over $\ProjC{1}$ is isomorphic to $\OO(d_1)\oplus\cdots\oplus \OO(d_n)$ for some $d_1 \geq d_2\geq \cdots \geq d_n \in \Z^n$. Moreover, the tuple $(d_1, \ldots, d_n)$ is unique.
\end{theo}

We call $\overrightarrow{\deg}(\EE)=(d_1,\ldots, d_n)$ the \emph{degree vector} of $\EE$. Note that the degree of $\EE$ is 
\[\deg \EE = \deg \Lambda^n \EE = d_1+ \cdots + d_n~.\]

The degree vector defines a map
\[\function{(\overrightarrow{\deg}_1,\overrightarrow{\deg}_2)}{\mathcal{M}(\alpha,\beta)}{\Z^p\times \Z^q}{\left(\UU_\bullet\oplus \VV_\bullet, \gamma\oplus \delta\right)}{ (\overrightarrow{\deg}~\UU, \overrightarrow{\deg}~\VV)~.}\]

\noindent The components $\overrightarrow{\deg}_1$ and $\overrightarrow{\deg}_2$ are only upper semi-continuous for the lexicographic order on $\Z^p$ and~$\Z^q$. Their level sets define a stratification of $\mathcal{M}(\alpha,\beta)$.

\section{Geometric Invariant Theory and quiver varieties}\label{s:GIT}

In this section, we recall a few facts about Mumford's Geometric Invariant Theory \cite{GIT} and apply it to the construction of certain quiver varieties called \emph{Kronecker varieties}, as well as a mix between these and spaces of flag configurations, that we call \emph{feathered Kronecker varieties}. These varieties will arise in our construction of compact components in the next sections.

\subsection{GIT quotients of complex varieties}

Let $Y$ be a smooth quasi-projective variety with an algebraic action of a complex reductive algebraic group $G$. The goal of Geometric Invariant Theory (often abreviated in GIT) is to define a good notion of quotient of $Y$ under the action of $G$ in the category of complex algebraic varieties. We briefly recall here Mumford's construction, which relies on the choice of an ample $G$-linearized line bundle. We restrict for simplicity to the case where $Y$ is projective over an affine base. (This will be sufficient for our applications.)

We do not want to assume the action to be faithful. Let us denote by $Z$ the kernel of the action, that is, the subgroup of $G$ that acts trivially on $Y$.

\begin{defi} \label{d:GLinearization}
A \emph{$G$-linearized line bundle} $L$ is a line bundle over $Y$ together with an algebraic $G$-action on the total space of $L$ that lifts the action on $Y$, and such that $Z$ acts trivially on $L$.

\end{defi}

Let $L$ be a $G$-linearized ample line bundle and $k$ an integer. The group $G$ acts on the space of sections of $L^k$. Denote by $H^0(L^k)^G$ the subspace of $G$-invariant sections of $L^k$. The tensor product of sections defines a product $H^0(L^k)^G\times H^0(L^l)^G \to H^0(L^{k+l})^G$, endowing the space
\[R(L)^G = \bigoplus_{k=0}^{+\infty} H^0(L^k)^G\]
with the structure of a graded algebra.

\begin{defi}
The GIT quotient of the polarized variety $(Y,L)$ by $G$ is the projective scheme
\[Y^L\sslash G := \mathrm{Proj}\left(R(L)^G\right)~.\]
\end{defi}

Note that $Y^L\sslash G$ is a scheme over $\mathrm{Spec}\left(H^0(\mathcal{O})^G\right)$, where $H^0(\mathcal{O})^G$ is the algebra of $G$-invariant functions on $Y$.

\smallskip
The main discovery of Mumford is that this algebraic construction has a geometric interpretation: the (closed) points of $Y^L\sslash G$ ``almost'' parametrize the $G$-orbits in $Y$.

\begin{defi}
A point $x$ in $(Y,L)$ is called
\begin{itemize}
\item \emph{semistable} if there exists $m>0$ and a $G$-invariant section $s$ of $L^m$ such that $s(x)\neq 0$,
\item \emph{polystable} if it is semistable and its $G$-orbit is closed in the subset of semistable points,
\item \emph{stable} if it is polystable and its stabilizer is finite modulo $Z$,
\item \emph{unstable} if it is not semistable.
\end{itemize}
\end{defi}

Let $Y^{ss}(L)$ denote the space of semistable points in $(Y,L)$. $Y^{ss}(L)$ is clearly open in $Y$ for the Zariski topology (in particular, it is dense when non-empty). Let $Y^{ss}(L)\sslash G$ denote the quotient of $Y^{ss}(L)$ by the equivalence relation where $x\sim y$ if the closures of the orbits of $x$ and $y$ in $Y^{ss}(L)$ intersect.

\begin{theo}[Mumford]
The topological space $Y^{ss}(L)\sslash G$ is homeomorphic to the variety $Y^L\sslash G$.
\end{theo}

The semistability of a point $x$ can be checked by looking at the action of one parameter subgroups of $G$, thanks to the \emph{Hilbert--Mumford criterion}. This criterion is often stated for projective varieties, but it has been extended to other settings. We present here an \emph{ad hoc} version adapted from \cite[Corollary~1.1]{RelativeHilbertMumford}.

Let $\lambda : \C^* \to G$ be a $1$-parameter subgroup and $x$ a point in $Y$. Let $x_0$ denote the limit $\lim_{t\to 0} \lambda(t)\cdot x$, when it exists. Then $x_0$ is fixed by $\lambda(\C^*)$ and $\lambda(\C^*)$ thus acts linearly on the fiber $L_{x_0}$ of $L$ at $x_0$.

\begin{defi}
The \emph{Hilbert--Mumford weight} $\mu_L(\lambda, x)$ is the integer $m$ such that
\[\lambda(t)\cdot v = t^{-m} v\]
for all $v\in L_{x_0}$ and all $t\in \C^*$. If the limit $\lim_{t\to 0} \lambda(t)\cdot x$ does not exist, we set $\mu_L(\lambda,x) = +\infty$.
\end{defi}

\begin{theo}[Hilbert--Mumford criterion] \label{t:HilbertMumford}
Assume that $Y = Y_1\times Y_2$ where $Y_1$ is affine and $Y_2$ is projective, and that the $G$-action on $Y$ is induced by algebraic $G$-actions on $Y_1$ and $Y_2$. Let $L$ be a $G$-linearized ample line bundle on $Y$. Then a point $x$ in $(Y,L)$ is semistable if and only if 
\[\mu_L(\lambda,x) \geq 0\]
for every $\lambda: \C^* \to G$. It is stable if the inequality is strict unless $\lambda$ is trivial.
\end{theo}

This stability criterion will prove particularly handy when we look at GIT quotients of products. Indeed, we have the following proposition:

\begin{prop} \label{p:MuProduct}
Let $(Y_i, L_i)_{1\leq i\leq n}$ be a family of quasi-projective varieties with an algebraic $G$-action and a $G$-linearized ample line bundle. Let $Y$ denote the product $\prod_{i=1}^n Y_i$ with the diagonal action of $G$. Let $p_i: Y\to Y_i$ denote the projection to the $i$-th factor and let $L$ be the $G$-linearized line bundle on $Y$ defined as
\[L= \bigotimes_{i=1}^n p_i^* L_i~.\]
Then for every $x = (x_1,\ldots, x_n)$ in $Y$ and every $\lambda: \C^* \to G$, we have
\[\mu_L(\lambda,x) = \sum_{i=1}^n \mu_{L_i}(\lambda, x_i)~.\]
\end{prop}

\subsection{Kronecker varieties} \label{ss:KroneckerVarieties}

Let $p,q, r$ be positive integers. The actions of $\GL(p,\C)$ and $\GL(q,\C)$ on $\C^p$ and $\C^q$ induce a linear action of the group
\[G_{p,q} \equaldef \GL(p,\C) \times \GL(q,\C)\]
on the space
\[E(p,q,r) = \Hom(\C^p, \C^q)^r\]
given by
\[(g_1,g_2)\cdot (A_j)_{1\leq j \leq r} = (g_2\circ A_j \circ g_1^{-1})_{1\leq j \leq r}~.\]
The kernel of this action is the central subgroup $Z_{p,q}= \C^*(\I_p, \I_q)$.

\smallskip
Let $\chi: G_{p,q} \to \C^*$ denote the character defined by
\[\chi: (g_1,g_2) \mapsto \det(g_2)^{p'} \det(g_1)^{-q'}~,\]
where $p' = \frac{p}{\gcd(p,q)}$ and $q' = \frac{q}{\gcd(p,q)}$. This character induces an action of $G_{p,q}$ on the trivial bundle $E(p,q,r)\times \C$ over $E(p,q,r)$, with kernel $Z_{p,q}$, defined by
\[g\cdot (x,z) = (g\cdot x, \chi^{-1}(g) z)~.\]
This provides $E(p,q,r)\times \C$ with the structure of a $G_{p,q}$-linearized line bundle, that we denote by $L_\chi$. 

\begin{rmk}
Since every line bundle over $E(p,q,r)$ is trivial and since every character on $G_{p,q}$ with kernel containing $Z_{p,q}$ is a power of $\chi$, every $G_{p,q}$-linearized line bundle over $E(p,q,r)$ is a power of $L_\chi$.
\end{rmk}

We can now define a GIT quotient of the semistable locus of $E(p,q,r)$ under the action of $G_{p,q}$. This quotient is a particular case of quiver variety sometimes refered to as a Kronecker moduli space (see for instance \cite{EulerCharacteristicKronecker}). Here we call it a \emph{Kronecker variety}.

\begin{defi}\label{d:kroneckermodulispace}
The \emph{Kronecker variety} $\mathcal{R}(p,q,r)$ is the GIT quotient
\[E(p,q,r)^{L_{\chi}}\sslash G_{p,q}~.\]
\end{defi}
Note that the only $G_{p,q}$-invariant regular functions on $E(p,q,r)$ are constant (see \cite{King}). Thus $\mathcal{R}(p,q,r)$ is a (reduced) projective scheme over $\C$, that is, a projective variety.

\begin{rmk}
The Kronecker varieties $\mathcal{R}(p,q,r)$ are a particular case of \emph{quiver varieties}. More precisely, $\RR(p,q,r)$ is the representation variety of the Kronecker quiver :
\begin{figure}[!h] 
\begin{center}
\includegraphics[height=1.5cm]{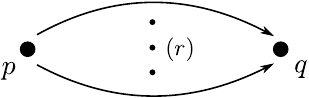}
\end{center}
\caption{The Kronecker quiver: $(r)$ indicates $r$ copies of the same oriented arrow while $p$ and $q$ indicate the dimensions of the vector spaces attached to the dots.} 
\label{f:kroneckerquiver}
\end{figure}
\end{rmk}

We will now apply the Hilbert--Mumford criterion (Theorem \ref{t:HilbertMumford}) to characterize (semi)stable points of $E(p,q,r)$. Let $\lambda: \C^* \to G_{p,q}$ be a $1$-parameter subgroup. Then we have
\[\lambda(e^t) = \left(\exp(tu), \exp(tv)\right)\]
where $u$ and $v$ are diagonalizable endomorphisms of $\C^p$ and $\C^q$ respectively, with integral eigenvalues. Let $m_1 > \cdots > m_k$ and $n_1 > \cdots >n_l$ denote the eigenvalues of $u$ and $v$ respectively, and let $F_i$ and $H_j$ denote the eigenspaces of $u$ and $v$ associated respectively to the eigenvalues $m_i$ and $n_j$, so that
\[\left\{ \begin{array}{l} 
u = m_1 \Id_{F_1} \oplus\cdots\oplus m_k \Id_{F_k} \\
v = n_1 \Id_{H_1}\oplus\cdots\oplus n_l\Id_{H_l}
\end{array} \right.\]
with respect to the decompositions $\C^p= F_1\oplus\cdots\oplus F_k$ and $\C^q=H_1 \oplus \cdots \oplus H_l$. We then define decreasing filtrations of $\C^p$ and $\C^q$ associated to $\lambda$.

\begin{defi}\label{d:filtrations}
The filtrations $(U_n)_{n \in \Z}$ and $(V_n)_{n\in \Z}$ associated to $\lambda$ are defined by
\[\left\{\begin{array}{lll}
U_n(\lambda) & = & \bigoplus_{m_i \geq n} F_i~, \\
V_n(\lambda) & = & \bigoplus_{n_j \geq n} H_j~.
\end{array}\right.\]
\end{defi}

\begin{prop} \label{p:muKronecker}
Let $\mathbf{A} = (A_1, \ldots, A_r)$ be a point in $E(p,q,r)$, $\lambda: \C^* \to G_{p,q}$ be a one parameter subgroup and $(U_n(\lambda))_{n\in\Z}$ and $(V_n(\lambda))_{n\in\Z}$ the associated filtrations. Then $\mu_{L_\chi}(\lambda, \mathbf{A})$ is finite if and only if $A_j(U_n(\lambda)) \subset V_n(\lambda)$ for all $n\in \Z$ and all $j\in \{1,\ldots, r\}$. In that case, we have
\[\mu_{L_\chi}(\lambda, \mathbf{A}) = \sum_{n\in \Z} p' \dim(V_n(\lambda)) - q' \dim(U_n(\lambda))~,\]
where $q'=\frac{q}{\gcd(p,q)}$ and $p'=\frac{p}{\gcd(p,q)}$.
\end{prop}

\begin{rmk}
The sum is well-defined because for $n>\max(m_1,n_1)$ we have $\dim(V_n(\lambda)) = \dim(U_n(\lambda)) = 0$ and for $n\leq \min(m_k,n_l)$, we have $\dim(U_n(\lambda)) = p$ and $\dim(V_n(\lambda)) = q$, so $p' \dim(V_n(\lambda))=q' \dim(U_n(\lambda))$.
\end{rmk}

\begin{proof}[Proof of Proposition \ref{p:muKronecker}]
The following computation is a particular case of a more general result of King for quiver varieties \cite{King}. We give it here for completeness.

\smallskip
Writing the $A_j$ with respect to diagonalization bases for $u$ and $v$, one easily verifies that $\lambda(t) \cdot \mathbf{A}$ converges when $t$ goes to $0$ if and only if $A_j(U_n(\lambda)) \subset V_n(\lambda)$ for all $j\in\{1,\ldots,r\}$ and all $n\in \Z$. Otherwise, there is a vector $w$ in some $F_{m_a}$ and some $j\in \{1,\ldots , r\}$ such that $A_j(w)$ has a non-zero component in $H_{n_b}$ for some $n_b < m_a$. Then $\lambda(t)\cdot A_j (w)$ grows at least like $t^{(n_b-m_a)}$ when $t$ goes to~$0$.

If $\lambda(t) \cdot \mathbf{A}$ converges as $t\to 0$, then the action of $\lambda(t)$ on the fiber of $L_\chi$ at the limit is simply given by multiplication by $\chi^{-1}(\lambda(t))$ and we thus obtain
\begin{equation} \label{eq:Chi(lambda)}
\mu_{L_\chi}(\lambda, \mathbf{A}) = p' \sum_{i=1}^l n_i \dim(H_i) - q' \sum_{i=1}^k m_i \dim(F_i)~.
\end{equation}
Let us take as a convention that $m_0 = n_0 = b > \max(m_1,n_1)$ and $m_{k+1} = n_{l+1} = a < \min(m_k,n_l)$. Note that we have $\dim(F_i) = \dim U_{m_i}(\lambda) - \dim U_{m_{i-1}}(\lambda)$, $\dim(H_j) = \dim V_{n_j}(\lambda) - \dim V_{n_{j-1}}(\lambda)$ for $1\leq i \leq k$ and $1\leq j \leq l$. Applying an Abel transform to \eqref{eq:Chi(lambda)}, we get

\begin{eqnarray*}
\mu_{L_\chi}(\lambda, \mathbf{A}) &=& p' \sum_{j=0}^l (n_j-n_{j+1}) \dim(V_{n_j}(\lambda)) - q' \sum_{i=0}^k (m_i-m_{i+1}) \dim(U_{m_i}(\lambda)) \\
&& + p'n_{l+1} \dim(V_{n_l}(\lambda)) - q' m_{k+1} \dim(U_{m_k}(\lambda))\\
&=& p' \sum_{j=0}^l \sum_{n= n_{j+1}+1}^{n_j} \dim(V_{n_j}(\lambda)) -q' \sum_{i=0}^k \sum_{n= m_{i+1}+1}^{m_i} \dim(U_{m_i}(\lambda)) + a(p' q - q'p) \\
&=& p' \sum_{j=0}^l \sum_{n= n_{j+1}+1}^{n_j} \dim(V_n(\lambda)) -q' \sum_{i=0}^k \sum_{n= m_{i+1}+1}^{m_i} \dim(U_n(\lambda)) \\
&=& \sum_{n=a}^b p' \dim(V_n(\lambda)) - q' \dim(U_n(\lambda))~.
\end{eqnarray*} 

\end{proof}
Proposition \ref{p:muKronecker} easily implies King's characterization of semi-stable points in $E(p,q,r)$ :

\begin{theo}[King, \cite{King}] \label{t:KingStability}
A point $(A_1, \ldots, A_r) \in E(p,q,r)$ is semistable for the action of $G_{p,q}$ if and only if all  subspaces $U$ and $V$ of $\C^p$ and $\C^q$ such that $A_j(U) \subset V$ for all $j\in\{1,\ldots,r\}$ satisfy
\[\dim V \geq \frac{q}{p}\dim U ~.\]
It is stable if, moreover, the equality only holds when $U$ and $V$ are reduced to $\{0\}$ or when $U= \C^p$ and $V= \C^q$. 
\end{theo}

\begin{coro}[see \cite{EulerCharacteristicKronecker}]{\ \\} \label{c:ExistenceStablePoint}
\begin{itemize}
\vspace{-0.4cm}
\item If $\frac{p}{q}+ \frac{q}{p} \geq r$, then $\mathcal{R}(p,q,r)$ is empty except when $r=1$ or $2$ and $p=q$.

\item If $\frac{p}{q}+ \frac{q}{p} < r$, then $E(p,q,r)$ contains stable points and $\mathcal{R}(p,q,r)$ has positive dimension.

\item If $p$ and $q$ are coprime, then every semistable point in $E$ is stable.

\end{itemize}
\end{coro}

The cases $r \leq 2$ and $p = q$ are special and are often excluded of the study of Kronecker varieties. However, they are relevant here. The following proposition deals with them.

\begin{prop}{\ } \label{p:R(p,p,2)}
\begin{itemize}
\item The Kronecker variety $\mathcal{R}(p,p,1)$ is reduced to a single point. There is no stable point. A point of $E(p,p,1) = \End(\C^p)$ is semistable if and only if it is invertible.

\item The Kronecker moduli space $\mathcal{R}(p,p,2)$ is isomorphic to $\ProjC{p}$. There is no stable point.
\end{itemize}
\end{prop}

The proof will rely on the following linear algebra lemma:
\begin{lem} \label{l:Semistability2Matrices}
Let $A$ and $B$ be two endomorphisms of $\C^p$. Then $(A,B)$ is semistable in $E(p,p,2)$ if and only if there exists $t\in \R$ such that $A+tB$ is invertible.
\end{lem}

\begin{proof}
Assume first that $A+tB$ is invertible for some $t$. If $V$ and $W$ are subspaces of $\C^p$ such that $A(V)\subset W$ and $B(V)\subset W$, then $A+tB(V)\subset W$. Hence $\dim W \geq \dim V$, proving semistability.

Conversely, assume that $A+tB$ has a kernel for all $t$. For all $n\in \N_{>0}$, let $u_n$ be a non-zero vector in the kernel of $A+\frac{1}{n}B$. Up to scaling $u_n$ and extracting a subsequence, we can assume that $u_n$ converges to a non-zero vector $u$ in the kernel of $A$. Set $V = \mathrm{Span}(u_n)_{n\in \N}$ and $W = A(V)$. Since $u_n$ is in the kernel of $A+\frac{1}{n}B$, $Bu_n$ is colinear to $A u_n$ for all $n$, hence $B(V)\subset W$. Finally, since $V$ is closed (as a linear subspace in finite dimension), it contains the vector $u$. Thus $A : V\to W$ is not injective and $\dim W < \dim V$, contradicting semistability.
\end{proof}

\begin{proof}[Proof of Proposition \ref{p:R(p,p,2)}]{\ \\}
\vspace{-0.4cm}
\begin{itemize}
\item  Case $r=1, p=q$.\\
If $A \in \End(\C^p)$ is not invertible, then the dimension of the image of $A$ is strictly less than $p$, which contradicts stability. If $A$ is invertible, then $\dim A(U) = \dim U$ for every subspace $U$ of $\C^p$, hence $A$ is semistable and not stable. Since any invertible matrice is similar to the identity, there is a unique semistable orbit. Hence $\mathcal{R}(1,p,p)$ is reduced to a point.

\smallskip
\item Case $r=2, p=q$.\\
By Lemma \ref{l:Semistability2Matrices}, there is a well-defined map
\[\function{\Phi}{E(p,p,2)^{ss}}{\mathbf{P}(\C_p[X,Y])}{(A_1,A_2)}{[\det(XA_1+YA_2)]}~,\]
where $\C_p[X,Y]$ is the space of homogeneous polynomials of degree $p$ in two variables. The map $\Phi$ is invariant under the action of $G_{p,q}$. It is a linear algebra exercise to verify that $\Phi$ is surjective and $\Phi(\mathbf A) = \Phi(\mathbf B)$ if and only if the closures of the $G_{p,q}$-orbits of $\mathbf A$ and $\mathbf B$ intersect. The map $\Phi$ thus induces the required isomorphism.

Finally, if $(A_1,A_2)$ is a semistable point, let $t$ be such that $A_t \equaldef A_1 + tA_2$ is invertible and let $g$ be an invertible matrix commuting with $A_t^{-1}A_2$. Then $(g,A_t g A_t^{-1}) \in G_{p,p}$ fixes $(A_1,A_2)$. Thus $(A_1,A_2)$ is not stable.
\end{itemize}
\end{proof}

\subsection{Flag configurations} \label{ss:FlagConfigurations}

Another classical application of GIT is the construction of configuration spaces of flags (see for instance \cite{FalbelMaculanSarfatti}).

Recall that a \emph{complete flag} $F$ in $\C^p$ is a nested sequence of linear subspaces 
\[\C^p= F_0 \supset F_1\supset \cdots \supset F_p = \{0\}~,\]
where $\dim F_i = p-i$. We denote by $\mathcal{F}(\C^p)$ the set of complete flags in $\C^p$.

\begin{rmk}
All the content of this subsection and the next one would easily extend to \emph{partial flags} (\textit{i.e.} decreasing sequences of subspaces). We chose to restrict to complete flags to avoid proliferation of indices. For the rest of this section, ``flag'' will mean ``complete flag''.
\end{rmk}

Let $\Gr_i(\C^p)$ denote the Grassmanian of $i$-dimensional subspaces of $\C^p$, and define
\[\function{\iota_i}{\FF(\C^p)}{\Gr_{p-i}(\C^p)}{(F_j)_{0\leq j\leq p}}{F_i}~.\]
Then the map $\iota = (\iota_1,\ldots, \iota_{p-1})$ embeds $\FF(\C^p)$ into $\prod_{i=1}^{p-1} \Gr_{p-i}(\C^p)$. 

Denote by $\OO_i(a)$ the pullback to $\Gr_i(\C^p)$ of the unique line bundle of degree $a$ on $\mathbf P\left(\bigwedge^i\C^p\right)$ by the Pl\"ucker embedding. The following is classical.

\begin{prop}
Every line bundle over $\mathcal{F}(\C^p)$ is isomorphic to
\[\OO(a_1,\ldots, a_{p-1}) \equaldef \bigotimes_{i=1}^{p-1} \iota_i^*\OO_{p-i}(a_i)\] for some $(a_1,\ldots, a_{p-1})\in\Z^{p-1}$. This bundle is ample if and only if all the $a_i$ are positive.
\end{prop}

The group $\GL(p,\C)$ acts on each $\Gr_i(\C^p)$ with kernel $\C^* \Id_p$. There is a canonical lift of this action to the total space of $\OO_i(1)$, such that $\lambda \Id_p$ acts by multiplication by $\lambda^{-i}$ in each fiber. This action, however, is not a $\GL(p,\C)$ linearization according to Definition \ref{d:GLinearization} since it does not have $\C^*\Id_p$ in its kernel. One can nonetheless linearize the action on $\OO_i(p)$. More precisely, if we simply denote by $g(v)$ the action of $g$ on a point $v\in \OO_i(p)$ induced by the canonical action of $\GL(p,\C)$ on $\OO_i(1)$, then we get a new action by setting
\[g\cdot v = \det(g)^i~g(v)\]
with kernel $\C^*\Id_p$, thus providing $\OO_i(p)$ with a $\GL(p,\C)$-linearization. The computation of the Hilbert--Mumford weights in this setting is classical.

\begin{prop} \label{p:MuGrassmanian}
Let $\lambda$ be a $1$-parameter subgroup of $\GL(p,\C)$ and let $U_n(\lambda)$ be the filtration defined in Definition \ref{d:filtrations}.
For $F$ in $\Gr_i(\C^p)$, we have
\[\mu_{\OO_i(p)}(\lambda, F) = \sum_{n\in \Z} i \dim(U_n(\lambda)) - p \dim(U_n(\lambda)\cap F)~.\]
\end{prop}

Now, let $s$ be a positive integer. We denote respectively by $\FF^s$ and $\HH^s$ the spaces $\FF(\C^p)^s$ and $\FF(\C^q)^s$.
The group $G_{p,q} = \GL(p,\C)\times \GL(q,\C)$ acts on $\FF^s \times \HH^s$ with kernel $\C^*\Id_p\times \C^*\Id_q$. One can endow $\FF^s\times \HH^s$ with many $G_{p,q}$-linearized ample line bundles via its embedding into
\[\Pi \equaldef \prod_{j=1}^s \left(\prod_{i=1}^{p-1} \Gr_{p-i}(\C^p) \times \prod_{i=1}^{q-1} \Gr_{q-i}(\C^q)\right)~.\]

More precisely, for any family of positive integers $\mathbf{a} = (a_i^j)_{1\leq j \leq s, 1\leq i \leq p-1}$ and $\mathbf{b} = (b_i^j)_{1\leq k \leq s, 1\leq i \leq q-1}$,
define the $G_{p,q}$-linearized line bundle $\OO(p \mathbf{a}, q \mathbf{b})$ as the restriction to $\FF^s \times \HH^s$ of the line bundle
\[\bigotimes_{j=1}^s \left(\pi_j^* \OO(pa_1^j,\ldots, pa_{p-1}^j) \otimes {\pi'_j}^*\OO(qb_1^j,\ldots, qb_{q-1}^j)\right)~,\]
where $\pi_j$ and $\pi'_j$ denote respectively the projections of $\Pi$ to the $j$-th factors $\FF(\C^p)$ and $\HH(\C^q)$. Then the actions of $\GL(p,\C)$ and $\GL(q,\C)$ on $\OO_{p-i}(p)$ and $\OO_{q-i}(q)$ described above induce a $G_{p,q}$-linearization of $\OO(p\mathbf a, q\mathbf b)$.

\smallskip
For $(\mathbf F,\mathbf H)= ((F_i^j), (H_i^j))\in \FF^s \times \HH^s$, and $\lambda = (\lambda_1,\lambda_2): \C^* \to G_{p,q}$ a $1$-parameter subgroup, write $\mu_{\mathbf{a}, \mathbf{b}}(\lambda,(\mathbf F,\mathbf H))$ for $\mu_{\OO(p\mathbf a,q\mathbf b)}(\lambda, (\mathbf F,\mathbf H))$. By Proposition \ref{d:GLinearization}, we have
\begin{equation} \label{eq:MuFlagConfiguration}
\mu_{\mathbf{a}, \mathbf{b}}(\lambda,(\mathbf F,\mathbf H)) = \sum_{j=1}^s \left(\sum_{i=1}^{p-1} a_i^j \mu_{\OO_{p-i}(p)}(\lambda_1, F_i^j) + \sum_{i=1}^{q-1} \mu_{\OO_{q-i}(q)}(\lambda_2, H_i^j)\right)~.
\end{equation}
We want to express this formula in terms of ``induced weights'' on the filtration associated to $\lambda$. In order to do so, let us choose $(\eta_i^j)_{1\leq j \leq s, 1\leq i \leq p}$ and $(\zeta_i^j)_{1\leq j \leq s, 1\leq i \leq q}$ such that
\[\eta_{i+1}^j - \eta_i^j = a_i^j\]
and
\[\zeta_{i+1}^j - \zeta_i^j = b_i^j\]
for all $i,j$. We see $\eta^j = (\eta_i^j)_{1\leq i \leq p}$ (resp. $\zeta^j = (\zeta_i^j)_{1\leq i \leq q}$) as weights attached to the flag $F^j = (F_i^j)_{0\leq i \leq p}$ (resp. $H^j = (H_i^j)_{0\leq i \leq q}$). If $U$ (resp. $V$) is a subspace of $\C^p$ (resp. $\C^q$), we denote by $\vert\eta^j(U \cap F^j)\vert$ (resp. $\vert\zeta^j(V \cap H^j)\vert$) the sum of the weights of the weighted flag induced on $U$ (resp. $V$) by $(F^j,\eta^j)$ (resp. $(H^j, \zeta^j)$), \emph{i.e.}
\[\vert\eta^j(U\cap F^j)\vert = \sum_{i=1}^{p} \eta_i^j \left( \dim U\cap F_{i-1}^j - \dim U \cap F_i^j\right)\]
and
\[\vert\zeta^j(V \cap H^j)\vert = \sum_{i=1}^{q} \zeta_i^j \left( \dim V\cap H_{i-1}^j - \dim V \cap H_i^j\right)~.\]
Finally, we set
\[\vert \eta(U\cap \mathbf F)\vert = \sum_{j=1}^s \vert\eta^j(U \cap F^j)\vert~,\]
\[\vert \zeta (V\cap \mathbf H)\vert = \sum_{j=1}^s \vert\zeta^j(V \cap H^j)\vert~,\]
\[\norm{\eta} = \sum_{j=1}^s \sum_{i=1}^{p} \eta_i^j  = \vert \eta(\C^p\cap \mathbf F)\vert\]
and 
\[\norm{\zeta} = \sum_{j=1}^s \sum_{i=1}^{q} \zeta_i^j =\vert \zeta(\C^q\cap \mathbf H)\vert~.\]

\begin{prop} \label{mu:FlagConfiguration}
For $(\mathbf F,\mathbf H) \in \FF^s \times \HH^s$ and $\lambda : \C^* \to G_{p,q}$ a $1$-parameter subgroup, we have
\begin{eqnarray*}
\mu_{\mathbf{a}, \mathbf{b}}(\lambda, (\mathbf F,\mathbf H)) &=& \sum_{n\in \Z} \big( \norm{\eta} \dim(U_n(\lambda)) - p \vert \eta(U_n(\lambda)\cap \mathbf F)\vert\\
&&\hphantom{\sum_{n\in \Z}} + \norm{\zeta} \dim(V_n(\lambda)) - q \vert \zeta (V_n(\lambda)\cap \mathbf H)\vert \big)~.
\end{eqnarray*}
\end{prop}

\begin{proof}
Replacing $\mu(\lambda, F_i^j)$ and $\mu(\lambda, H_i^j)$ in \eqref{eq:MuFlagConfiguration} by the formula given in Proposition \ref{p:MuGrassmanian}, we obtain
\begin{eqnarray*}
\mu_{\mathbf{a}, \mathbf{b}}(\lambda, (\mathbf F,\mathbf H)) &=& \sum_{n\in \Z} \sum_{j=1}^s \left(\sum_{i=1}^{p-1} a_i^j (p-i) \dim U_n(\lambda) - p a_i^j \dim U_n(\lambda)\cap F_i^j\right. \\
&& + \left. \sum_{i=1}^{q-1} b_i^j (q-i) \dim V_n(\lambda) - q b_i^j \dim V_n(\lambda) \cap H_i^j\right)~.
\end{eqnarray*}

Writing $a_i^j = \eta_{i+1}^j - \eta_i^j$, $b_i^j = \zeta_{i+1}^j - \zeta_i^j$ and applying an Abel transform, one obtains the result.
\end{proof}

\subsection{Feathered Kronecker varieties}
\label{ss:FeatheredKroneckerVarieties}
Let us define
\[E_\spadesuit(p,q,r,s) = E(p,q,r) \times \FF^s \times \HH^s~.\]
In this section, we consider GIT quotients of $E_\spadesuit(p,q,r,s)$ with different choices of $G_{p,q}$-linearized line bundles.

\smallskip
Let $\eta = (\eta_i^j)_{1\leq j \leq s, 1\leq i \leq p}$ and $\zeta = (\zeta_i^j)_{1\leq j \leq s, 1\leq i \leq q}$ be tuples of \emph{real} numbers such that $\eta_1^j <\cdots < \eta_p^j$ and $\zeta_1^j <\cdots < \zeta_q^j$ for all $j$. We first define a numerical stability condition on $E_\spadesuit(p,q,r,s)$ depending on $\eta$ and $\zeta$. We will then prove that this stability condition is equivalent to Mumford's stability for a suitable choice of ample $G_{p,q}$-linearized line bundle.

\begin{defi} \label{d:(epsilon,eta)stability}
A point $(\mathbf{A},\mathbf F,\mathbf H) \in E_\spadesuit(p,q,r,s)$ is \emph{$(\eta, \zeta)$-semistable} if, for every subspaces $U$ and $V$ of $\C^p$ and $\C^q$ such that $A_i(U) \subset V$ for all $i \in \{1,\ldots, r\}$, we have
\[\mu_{\eta,\zeta}(U,V) \equaldef \left(\frac{\norm{\eta}}{p}-q\right)\dim U - \vert \eta(U\cap \mathbf F)\vert + \left(\frac{\norm{\zeta}}{q} + p\right)\dim V - \vert \zeta(V\cap \mathbf H)\vert \geq 0~.\]
It is \emph{$(\eta,\zeta)$-stable} if, moreover, equality only holds for $U = V= \{0\}$ or $U=\C^p$ and $V=\C^q$. The \emph{feathered Kronecker variety} $\mathcal{R}_\spadesuit(p,q,r,s,\eta, \zeta)$ is the largest Hausdorff quotient of the set of $(\eta,\zeta)$-semistable points of $E_\spadesuit(p,q,r,s)$ under the action of $G_{p,q}$.
\end{defi}

The following proposition asserts that these spaces are indeed GIT quotients of $E_\spadesuit(p,q,r,s)$. In particular, $\mathcal{R}_\spadesuit(p,q,r,s,\eta, \zeta)$ is homeomorphic to a projective variety.

\begin{prop}
There exists an ample $G_{p,q}$-linearized line bundle $L$ on $E_\spadesuit(p,q,r,s)$ (depending on $\eta$ and $\zeta$) such that $L$-semistable (resp. stable) points are exactly $(\eta,\zeta)$-semistable (resp. stable) points. 
\end{prop}

\begin{proof}
Let us first remark the following: one can find \emph{rational} weights $(\eta', \zeta')$ close to $(\eta, \zeta)$ such that the $(\eta', \zeta')$-(semi)stable points are exactly the $(\eta, \zeta)$-(semi)stable points. Indeed, all the (semi)stability conditions in Definition \ref{d:(epsilon,eta)stability} for all points in $E_\spadesuit(p,q,r,s)$ are linear equations and inequalities in $(\eta, \zeta)$ with integral coefficients bounded by some constant. There are thus only finitely many such conditions, and one can find a rational point $(\eta',\zeta')$ close to $(\eta, \zeta)$ which satisfies the same exact conditions.

We can therefore assume without loss of generality that the weights $(\eta, \zeta)$ are rational. Let $k$ be a positive integer such that $(\frac{k}{p}\eta, \frac{k}{q}\zeta)$ is integral. Set 
\[a_i^j = \frac{k}{p}\eta_{i+1}^j - \frac{k}{q}\eta_i^j~,\quad b_i^j = \frac{k}{q}\zeta_{i+1}^j - \frac{k}{q}\zeta_i^j~.\]
Let $L_{\mathbf a,\mathbf b}$ denote the $G_{p,q}$-linearized line bundle over $\FF^s\times \HH^s$ defined in Section \ref{ss:FlagConfigurations}. Let also $L_\chi$ be the $G_{p,q}$-linearized line bundle over $E(p,q,r)$ defined in Section \ref{ss:KroneckerVarieties}. Finally let $L_{k,\mathbf a, \mathbf b}$ denote the $G_{p,q}$-linearized line bundle on $E_\spadesuit(p,q,r,s)$ defined by
\[L_{k,\mathbf a, \mathbf b} = p_E^*L_\chi^{k\gcd(p,q)}\otimes p_{\mathcal{F}\times \mathcal{H}}^*L_{\mathbf a, \mathbf b}~,\]
where $p_E$ and $p_{\mathcal{F}\times \mathcal{H}}$ respectively denote the projections to the factors $E(p,q,r)$ and $\FF^s \times \HH^s$. We now write $\mu_{k,\mathbf a,\mathbf b}$ for $\mu_{L_{k,\mathbf a, \mathbf b}}$.

\smallskip
Let $(\mathbf{A},\mathbf F,\mathbf H)$ be a point in $E_\spadesuit(p,q,r,s)$. Given a one parameter subgroup $\lambda: \C^* \to G_{p,q}$, putting together Propositions \ref{p:muKronecker}, \ref{p:MuGrassmanian} and \ref{p:MuProduct}, we obtain that $\mu_{k,\mathbf a,\mathbf b}(\mathbf{A},\mathbf F,\mathbf H) = + \infty$ unless $A^j(U_n(\lambda)) \subset V_n(\lambda)$ for all $n$ and all $j$, in which case
\begin{align*}
\mu_{k,\mathbf a,\mathbf b}(\lambda, (\mathbf{A},F,H)) =& \sum_{n\in \Z} \left(\frac{k}{p} \norm{\eta} \dim U_n(\lambda) - k \vert \eta(U_n(\lambda)\cap \mathbf F)\vert + \frac{k}{q}\norm{\zeta} \dim V_n(\lambda) - k \vert \zeta(V_n(\lambda) \cap \mathbf H)\vert \right. \\
& \phantom{\sum_{n\in \Z} \left(\right.} \left. \phantom{\frac{k}{p}} + kp \dim V_n(\lambda) - k q \dim U_n(\lambda) \right)~.
\end{align*}

If $(\mathbf{A},\mathbf F,\mathbf H)$ is not $L_{k,\mathbf a,\mathbf b}$-semistable, then there exists a one parameter subgroup $\lambda$ such that $\mu_{k,\mathbf a,\mathbf b}(\lambda, (\mathbf{A},\mathbf F,\mathbf H)) <0$. Thus, for some $n\in \Z$, 
\[p \dim(V_n(\lambda)) - q \dim(U_n(\lambda)) +  \frac{\norm{\eta}}{p} \dim V_n(\lambda) - \vert \eta(V_n(\lambda)\cap \mathbf F)\vert + \frac{\norm{\zeta}}{q} \dim V_n(\lambda) - \vert \zeta(V_n(\lambda) \cap \mathbf H)\vert<0~,\]
contradicting $(\eta, \zeta)$-stability.

Conversely, if $(\mathbf{A},\mathbf F,\mathbf H)$ is not $(\eta, \zeta)$-semistable, one can find $U\subset \C^p$ and $V\subset \C^q$ such that 
\[(*) = p \dim(V) - q \dim(U) +  \frac{\norm{\eta}}{p} \dim U - \vert \eta (V\cap \mathbf F)\vert + \frac{\norm{\zeta}}{q} \dim V - \vert \zeta(W\cap \mathbf H)\vert<0~.\]
Define $\lambda(e^t) = \left(\exp(t p_U),\exp(t p_V)\right)$, where $p_U$ and $p_V$ are respectively projectors on $U$ and $V$. Then 
\[U_n(\lambda) = \left \vert \begin{array}{l} \{0\} \quad \textrm{if $n\leq 0$} \\
																		U \quad \textrm{if $n=1$}\\
																		\C^p \quad \textrm{if $n\geq 2$} \end{array} \right.\]
and
\[V_n(\lambda) = \left \vert \begin{array}{l} \{0\} \quad \textrm{if $n\leq 0$} \\
																		V\quad \textrm{if $n=1$}\\
																		\C^q \quad \textrm{if $n\geq 2$} \end{array}\right.~.\]
Thus $\mu_{k,\mathbf a, \mathbf b}(\lambda (\mathbf A, \mathbf F,\mathbf H)) = (*) < 0$, contradicting $L^{k,\mathbf a, \mathbf b}$-stability. 

This proves the equivalence between $(\eta, \zeta)$-semistability and $L_{k,\mathbf a, \mathbf b}$-semistability. Weakening the inequalities in the proof gives the equivalence between $(\eta, \zeta)$-stability and $L_{k,\mathbf a, \mathbf b}$-stability.
\end{proof}

Despite being given by an explicit formula, the $(\eta,\zeta)$-stability criterion is not very enlightening in general. For future use, let us specialize to the case where $\eta$ and $\zeta$ are small.

\begin{coro} \label{c:FeatherStabilitySmallEtaZeta}
For $(\eta, \zeta)$ sufficiently close to $(0,0)$, a point $(\mathbf A, \mathbf F,\mathbf H) \in E_\spadesuit(p,q,r,s)$ is $(\eta, \zeta)$-semistable if and only if, for every proper subspaces $U$ and $V$ of $\C^p$ and $\C^q$ such that $A^j(U) \subset V$ for all $j$,
\begin{itemize}
\item either $\dim V > \frac{q}{p} \dim U$
\item or $\dim V = \frac{q}{p} \dim U$ and 
\begin{equation}\label{eq:SimplifiedStability}
\vert \eta(U\cap \mathbf F)\vert + \vert\zeta(V\cap \mathbf H)\vert \leq \left(\norm{\eta} + \norm{\zeta}\right) \frac{\dim V}{q}~.
\end{equation}
\end{itemize}
It is stable if it is semistable and equality in \eqref{eq:SimplifiedStability} only holds for $U = V= \{0\}$ and $U= \C^p$, $V=\C^q$.

In particular, if $p$ and $q$ are coprime, then $\mathcal{R}_\spadesuit(p,q,r,s, \eta, \zeta)$ fibers over $\mathcal{R}(p,q,r)$ with fibers isomorphic to $\mathcal{F}^s \times \mathcal{H}^s$.
\end{coro}

\begin{proof}
For $(\eta,\zeta)$ small enough, $\mu_{\eta,\zeta}(U,V)$ has the same sign as $q\dim V - p \dim U$ when the latter is non-zero. If $q\dim U - p \dim V=0$, we have
\begin{eqnarray*}
\mu_{\eta, \zeta}(U,V) &=& \frac{\dim U}{q}\norm{\eta} + \frac{\dim V}{p} \norm{\zeta} -\vert \eta(U\cap \mathbf F)\vert - \vert \zeta(V\cap \mathbf H)\vert \\
&=& \frac{\dim V}{q}\left(\norm{\eta} +\norm{\zeta}\right) - \vert \eta(U\cap \mathbf F)\vert - \vert \zeta(V\cap \mathbf H)\vert~.\end{eqnarray*}
The first part of the proposition follows.

\smallskip
Assume now that $p$ and $q$ are coprime. Then we cannot have $q\dim U - p \dim V=0$ for proper subspaces $U$ and $V$. Hence, for $(\eta, \zeta)$ small enough, $(\mathbf{A}, \mathbf F,\mathbf H)$ is $(\eta,\zeta)$-stable if and only if $\mathbf{A}\in E(p,q,r)$ is stable. We thus have
\[E_\spadesuit^{\textit{stable}}(p,q,r,s, \eta, \zeta) = E(p,q,r)^{\textit{stable}} \times \mathcal{F}^s \times \mathcal{H}^s\]
and the projection on the first factor induces a morphism from $\mathcal{R}_\spadesuit (p,q,r,s, \eta,\zeta)$ to $\mathcal{R}(p,q,r)$. The fact that this map is a fibration with fibers isomorphic to $\mathcal{F}^s \times \mathcal{H}^s$ follows from the fact that the stabilizer of any stable point $\mathbf{A} \in E(p,q,r)$ is the subgroup $\C^*(\Id_p,\Id_q)$, which acts trivially on $E_\spadesuit(p,q,r,s)$.
\end{proof}

\section{Compact components}\label{s:mainsection}

\subsection{A compactness criterion} 

In this subsection, we find a sufficient criterion for a relative component $\MM(\alpha,\beta,d)$ of parabolic $\SU(p,q)$-Higgs bundles over $X=\ProjC{1}\setminus\{x_1,\ldots,x_s\}$ to be compact.

\smallskip
Recall from Subsection \ref{s:su(p,q)higgsbbundles} that a parabolic $\SU(p,q)$-Higgs bundle of type $(\alpha,\beta)=\big\{(\alpha^1,\beta^1),\ldots,(\alpha^s,\beta^s)\big\}$ is a pair $(\UU_\bullet\oplus \VV_\bullet,\gamma\oplus\delta)$ where:
\begin{itemize}
\item $\UU_\bullet$ and $\VV_\bullet$ are parabolic bundles of respective rank $p$ and $q$, with respective parabolic structures of type $\alpha^k$ and $\beta^k$ at $x_k\in\{x_1,\ldots,x_s\}$ and such that $\det(\UU_\bullet)\otimes \det(\VV_\bullet) =\mathcal{O}$.
\item $\gamma\in H^0\big(\KK(D)\otimes\Hom(\UU_\bullet,\VV_\bullet)\big)$ and $\delta\in \Hom\big(\KK(D)\otimes\Hom(\VV_\bullet,\UU_\bullet)\big)$, where $D$ is the effective divisor $x_1+\cdots+x_s$.
\end{itemize}

\begin{prop}\label{p:gammavanishes}
Let $(\alpha,\beta)$ be a $\SU(p,q)$-multiweight satisfying 
\begin{itemize}
\item[(1)] $\quad\alpha_p^j < \beta_1^j\quad \textrm{for all }j\in \{1,\ldots, s\}$,
\item[(2)] $\quad\epsilon \equaldef \sum_{j=1}^s \beta_q^j-\alpha_1^j < 2$.
\end{itemize}
Let $(\UU_\bullet \oplus \VV_\bullet,\gamma\oplus \delta)$ be a semistable parabolic Higgs bundle of type $(\alpha,\beta)$. If $\deg(\UU_\bullet)-\deg(\VV_\bullet) < 2- \epsilon$, then $\delta$ vanishes identically.
\end{prop}

\begin{rmk}
Since we only consider strongly parabolic Higgs bundles, the assumption that $\alpha^j_p \leq \beta^j_1$ would be sufficient to deduce the vanishing of $\delta$. We wrote a strict inequality because we will use later that condition~(1) is open.
\end{rmk}

\begin{proof}
The proof is similar to the Higgs bundle proof of the Milnor--Wood inequality (see \cite{BradlowGarciaPradaGothen}), where a violation of the Milnor--Wood inequality for $\deg(\UU_\bullet) - \deg(\VV_\bullet)$ would imply the vanishing of one of the terms of the Higgs field (contradicting stability).

\smallskip
The condition $\alpha_p^j<\beta_1^j$ implies that $\delta$ does not have poles at the puncture $x_j$, \emph{i.e.} $\delta$ is a holomorphic section of $\Hom(\VV,\UU\otimes \KK)$. Let $N$ and $I\otimes \KK$ be the subsheaves of $\VV$ and $\UU\otimes \KK$ respectively given by the kernel and the image of $\delta$, so that $\delta$ induces an exact sequence of sheaves
\[ 0 \to N \to V \to I\otimes \KK \to 0.\]
In particular, we have
\begin{equation}\label{e:compactnesscriterion1}
\deg(V)=\deg(N)+\deg(I\otimes \KK).
\end{equation}
Let $\NN\subset \VV$ and $\II\subset \UU$ be respectively the saturation of the sheaves $N$ and $I$. Denote by $n$ the rank of $\NN$. We have $\deg(N)\leq\deg(\NN)$ and $\deg(I)\leq \deg(\II)$. Equation (\ref{e:compactnesscriterion1}) implies
\begin{equation}\label{e:compactnesscriterion2}
\deg(\VV)\leq\deg(\NN)+\deg(\II) -2 (q-n).
\end{equation}
The bundles $\NN_\bullet$ and $\VV_\bullet\oplus \II_\bullet$ are $\gamma \oplus \delta$-invariant. By semistability, they must have non-positive parabolic degree. To write this more explicitly, let us define
\[\left\{\begin{array}{lll}
A(m) & = & \sum_{j=1}^s \sum_{i=1}^m \alpha_i^j\\
B(m) & = & \sum_{j=1}^s \sum_{i=1}^m \beta_{q+1-i}^j~.\end{array}\right.\]
In particular $A(p)=\Vert\alpha\Vert$ and $B(q)=\Vert\beta\Vert$.

The parabolic weights on $\NN_\bullet$ (resp. $\II_\bullet$) at the point $x_j$ are larger or equal to $\beta^j_1,\ldots, \beta^j_n$ (resp. $\alpha^j_1, \ldots, \alpha^j_{q-n}$). The stability condition thus implies
\begin{equation}\label{e:6}
\left\{\begin{array}{lll} \deg(\NN)+ B(q) - B(q-n) \leq 0 \\
\deg(\VV)+\deg(\II)+ B(q) + A(q-n) \leq 0~.\end{array}\right.
\end{equation}
Adding the two lines in \eqref{e:6} and using Equation \eqref{e:compactnesscriterion2}, we get
\begin{equation}\label{compactnesscriterion3}
2\deg(\VV) + 2 B(q) +2(q-n) + A(q-n) - B(q-n) \leq 0.
\end{equation}
Since $\deg(\VV_\bullet) = - \deg(\UU_\bullet) = \deg(\VV) + B(q)$, we can rewrite this as
\[\deg(\UU_\bullet)-\deg(\VV_\bullet)\geq 2(q-n) - (B(q-n)-A(q-n))~.\]
Finally, by definition of $A(q-n)$ and $B(q-n)$, we have
\begin{eqnarray*}
B(q-n)-A(q-n) & = & \sum_{j=1}^s \sum_{i=1}^{q-n} \beta_{q+1-i}^j -\alpha_i^j \\
& \leq &\sum_{j=1}^s (q-n)(\beta_q^j-\alpha_1^j)\\
&= &(q-n)\epsilon~.
\end{eqnarray*}
We thus obtain
\[\deg(\UU_\bullet)-\deg(\VV_\bullet)\geq (q-n)(2 - \epsilon)~.\]

If $\epsilon < 2$ and $\deg(\UU_\bullet)-\deg(\VV_\bullet) < 2-\epsilon$, then clearly this inequality only holds for $n= q$, which means that $\NN= \ker \delta = \VV$. 
\end{proof}

Consider $(\alpha, \beta)$ satisfying the conditions of Proposition \ref{p:gammavanishes} and, as before, denote by $\MM(\alpha,\beta,d)$ the moduli space of polystable parabolic $\SU(p,q)$-Higgs bundles of type $(\alpha,\beta)$ such that $\deg(\UU)-\deg(\VV)=d$.

If $(\UU_\bullet \oplus \VV_\bullet,\gamma\oplus \delta)$ is a stable Higgs bundle in $\MM(\alpha,\beta,d)$, then $\delta=0$ by the previous proposition. In particular, $\VV_\bullet$ is $\gamma\oplus\delta$-invariant and $\deg(\VV_\bullet)< 0$ by stability. Using $\deg(\UU_\bullet)+\deg(\VV_\bullet)=0$, stability gives $d>\Vert \beta\Vert-\Vert\alpha\Vert$. This provides a necessary condition to have stable point in $\MM(\alpha,\beta,d)$. 

Considering the interval
\[J_{\alpha,\beta}:=\big( \norm{\beta} - \norm{\alpha}, \norm{\beta} - \norm{\alpha} +2 -\epsilon \big)~,\] 
then Proposition \ref{p:gammavanishes} gives the following compactness criterion:

\begin{coro}\label{c-compactnessscriterion}
Let $(\alpha,\beta)$ be a $\SU(p,q)$-multiweight satisfying the conditions of Proposition \ref{p:gammavanishes}. If $d$ is an integer in $J_{\alpha, \beta}$, then $\MM(\alpha,\beta,d)$ is compact.
\end{coro}

\begin{proof}
Let $(\UU_\bullet\oplus \VV_\bullet, \gamma\oplus \delta)$ be a point in $\MM(\alpha,\beta,d)$. By definition of $J_{\alpha,\beta}$, one has
\[\deg(\UU_\bullet)-\deg(\VV_\bullet) = \norm{\alpha} - \norm{\beta} + d < 2- \epsilon~.\]
Therefore, $\delta$ vanishes identically by Proposition \ref{p:gammavanishes}. In particular, the Higgs field $\gamma\oplus \delta$ is nilpotent and the Hitchin map sends $\MM(\alpha,\beta,d)$ to $0$. The properness of the Hitchin map (see Theorem \ref{t:PropernessHitchinMap}) thus implies that $\MM(\alpha,\beta,d)$ is compact.
\end{proof}

\subsection{A non trivial example}\label{ss:ConstantType}

When $\epsilon <1$, the interval $J_{\alpha, \beta}$  has length at least $1$ and thus contains an integer. It is not easy, however, to construct a stable Higgs bundle with $\deg(\UU)- \deg(\VV) = d$ for generic $(\alpha, \beta)$. Here, we restrict to more specific choices of weights that are ``constant'' at each puncture, for which computations are easier. This will allow us to construct examples of moduli spaces $\mathcal{M}(\alpha,\beta,d)$ that are compact and isomorphic to the Kronecker moduli space $\mathcal{R}(p,q,s-2)$. In the next subsection, we will look at generic choices of weights close to these constant weights.

\smallskip
We will say that a multiweight $(\alpha,\beta)$ is \emph{constant} if for all $j\in \{1,\ldots, s\}$,
 \[\left\{ \begin{array}{l} \alpha_1^j  = \cdots = \alpha_p^j = \alpha^j\\
 \beta_1^j = \cdots = \beta_q^j = \beta^j~.\end{array}\right.\]
In this section, all multiweights $(\alpha,\beta)$ are assumed to be constant. We write
 \[\left\{ \begin{array}{l} \vert\alpha \vert = \sum_{j=1}^s \alpha^j \\
\vert\beta \vert = \sum_{j=1}^s \beta^j~,\end{array}\right.\]
so that $\norm{\alpha} = p \vert\alpha \vert,~\norm{\beta} = q \vert\beta\vert$ and $\epsilon = \vert\beta \vert - \vert\alpha\vert$.

Recall also that, by Birkhoff-Grothendieck's Theorem (see Theorem \ref{t:Birkhoff-Grothendieck}), any rank $n$ vector bundle $\EE$ over $\ProjC{1}$ decomposes holomorphically as $\EE=\OO(d_1)\oplus\cdots\oplus \OO(d_n)$ where $d_1\geq\cdots\geq d_n$, and $\overrightarrow{\deg}(\EE)=(d_1,\ldots,d_n)$ is the degree vector of $\EE$. 

\begin{lem}\label{l:decompositionU}
Let $(\alpha,\beta)$ be a constant multiweight satisfying the following conditions: $\alpha^j<\beta^j$ for all $j\in\{1,\ldots,s\}$ and $\epsilon \equaldef \vert\beta \vert - \vert \alpha \vert <2$. If $(\UU_\bullet\oplus \VV_\bullet,\gamma\oplus\delta)\in\MM(\alpha,\beta,d)$ for $d\in J_{\alpha,\beta}$, then $\overrightarrow{\deg}(\VV)=(-a,\ldots,-a)$ for some $a>0$.
\end{lem}

\begin{proof}
Let $(\UU_\bullet\oplus \VV_\bullet,\gamma\oplus\delta)\in\MM(\alpha,\beta,d)$ be such a parabolic $\SU(p,q)$-Higgs bundle.

By Remark \ref{r-degreeU}, $\deg(\VV)=-\frac{1}{2}(p\vert\alpha\vert + q\vert\beta\vert + d)$. Using $d<q\vert\beta\vert-p\vert\alpha\vert +2-\epsilon$, we get
\[\deg \VV > - \left(q |\beta| +1 - \frac{\epsilon}{2}\right)~.\]
By semistability, $\deg(\VV)< 0$, so we can write $\deg(\VV)=-aq+r$ with $a\in \N,~a>0$ and $r\in\{0,\ldots,q-1\}$. The previous equation implies
\begin{equation}\label{e:linedecomposition1}
-\vert \beta\vert< \frac{\deg(\VV)}{q}+\frac{1-\frac{1}{2}\epsilon}{q}=-a + \frac{r+1-\frac{1}{2}\epsilon}{q}.
\end{equation}
Let $\MM\subset \VV$ be a line bundle of maximal degree $m$. By semistability, $m+\vert\beta\vert\leq 0$, so Equation (\ref{e:linedecomposition1}) gives
\[m\leq  -\vert \beta\vert < -a + \frac{r+1-\frac{1}{2}\epsilon}{q}.\]
Using $\frac{r+1-\frac{1}{2}\epsilon}{q}<1$, we get $m\leq -a$. By maximality of $m$ we get $r=0$, $m=-a$ and $\VV = \OO(-a) \oplus \cdots \oplus \OO(-a)$. 
\end{proof}

The previous lemma shows that the assumptions on the triple $(\alpha,\beta,d)$ in Corollary \ref{c-compactnessscriterion} are rather restrictive. Nonetheless, we will be able to find $(\alpha,\beta,d)$ satisfying these conditions and for which $\mathcal{M}(\alpha,\beta,d)\neq \emptyset$. 

\begin{lem}\label{l:decompositionV}
Given an integer $a\in\left[\frac{s+p}{p+q},\frac{(p+q-1)s+p}{p+q}\right]$, there exists a constant multiweight $(\alpha,\beta)$ satisfying 
\begin{itemize}
\item[(a)] $\alpha^j<\beta^j$ for all $j\in\{1,\ldots,s\}$,
\item[(b)] $\epsilon \equaldef \vert\beta \vert - \vert\alpha\vert \in \left( \frac{2pq-2p-2q}{2pq-p-q},1\right)$,
\item[(c)] $d \equaldef(q-p)a+p\in J_{\alpha,\beta}$.
\end{itemize}
Moreover, if $(\UU_\bullet\oplus \VV_\bullet,\gamma\oplus\delta)\in \MM(\alpha,\beta,d)$, then $\overrightarrow{\deg}(\UU)=(-a+1,\ldots,-a+1)$.
\end{lem}

\begin{rmk}
The interval $\left[\frac{s+p}{p+q},\frac{(p+q-1)s+p}{p+q}\right]$ has length $\left(1- \frac{2}{p+q}\right)s$ which is at least $1$ as soon as $s\geq 3$ and $p$ or $q$ is greater than $1$. In that case this inteval always contains an integer. On the other hand, when $p = q =1$, it is reduced to $\{\frac{s+1}{2}\}$, which is an integer if and only if $s$ is odd (this corresponds to Example \ref{ex:p=q=1}).
\end{rmk}
\begin{proof}[Proof of Lemma \ref{l:decompositionV}]
The condition $a\in\left[\frac{s+p}{p+q},\frac{(p+q-1)s+p}{p+q}\right]$ is equivalent to
\begin{equation} \label{eq:Inequalitya}
s\leq (p+q)a-p \leq  (p+q-1) s~.
\end{equation}
Since $a$ is an integer, it implies $a\leq s$.

Write $(p+q)a -p = k s + r$ with $0\leq r< s$ and define $k^j = k+1$ for $j\in \{1,\ldots, r\}$ and $k^j = k$ for $j\in \{r+1,\ldots, s\}$, so that $\sum_{j=1}^s k^j = (p+q)a - p$. The inequality \eqref{eq:Inequalitya} implies that $1\leq k^j \leq p+q-1$ for all $j$.

We now choose $(\epsilon^j)_{1\leq j \leq s} \in \R_{>0}$ satisfying the following conditions:
\begin{equation} \label{eq:ConditionsEpsiloni}
\left \{
\begin{array}{l}
\epsilon^j < \min\left(\frac{k^j}{q},\frac{p+q-k^j}{p}\right) \quad \textrm{for all $j$}\\
\frac{2pq-2p-2q}{2pq-p-q} < \sum_{j=1}^s \epsilon^j < 1~.
\end{array}\right.
\end{equation}
Let us first show that these conditions can be simultaneously satisfied. If $k \leq q-1$, then $\min\left(\frac{k^j}{q},\frac{p+q-k^j}{p}\right) = \frac{k^j}{q}$, and the conditions \eqref{eq:ConditionsEpsiloni} can be simultaneously satisfied since
\[\sum_{j=1}^s \frac{k^j}{q} = \frac{(p+q)a-p}{q} \geq 1~.\]
Similarly, if $k\geq q$, we have $\min\left(\frac{k^j}{q},\frac{p+q-k^j}{p}\right) = \frac{p+q-k^j}{p}$, and the conditions \eqref{eq:ConditionsEpsiloni} can be simultaneously satisfied since
\[\sum_{j=1}^s \frac{p+q-k^j}{p} = \frac{(p+q)(s-a)+p}{p} \geq 1 \quad \textrm{(using $a\leq s$)}~.\\ \]

With the $k^j$ and $\epsilon^j$ chosen as above, let us set
\[\left\{\begin{array}{lll} 
\alpha^j & = & \frac{k^j - q\epsilon^j}{p+q} \\
\beta^j & = & \frac{k^j + p \epsilon^j}{p+q}~. \end{array}\right.\]
Then $p\alpha^j+q\beta^j = k^j \in \N$, $0<\alpha^j<\beta^j< 1$ for all $j\in\{1,\ldots,s\}$ and $\epsilon = \sum_{j=1}^s \epsilon^j \in \left(\frac{2pq-2p-2q}{2pq-p-q},1 \right)$, so $(\alpha,\beta)$ is a constant $\SU(p,q)$-multiweight and satisfies conditions (a) and (b).

With these choices, we have
\[\left\{ \begin{array}{lll}
\vert\alpha\vert & = & a - \frac{p+q\epsilon}{p+q} \\
\vert\beta\vert & = & a - \frac{p-p\epsilon}{p+q}~,\end{array}\right.\]
and thus
\[q\vert \beta \vert-p\vert\alpha\vert=(q-p)a-\frac{p(q-p)}{p+q}+\frac{2pq}{p+q}\epsilon,\]
the condition $(q-p)a+p\in J_{\alpha,\beta}$ is equivalent to the condition
\[-\frac{p(q-p)}{p+q}+\frac{2pq}{p+q}\epsilon<p<-\frac{p(q-p)}{p+q}+2+\frac{2pq-p-q}{p+q}\epsilon, \]
which is equivalent to $\epsilon \in \left(\frac{2pq-2p-2q}{2pq-p-q},1 \right)$. In particular, (c) is satisfied.

\smallskip
Consider now $(\UU_\bullet\oplus \VV_\bullet,\gamma\oplus \delta)\in \MM(\alpha,\beta,d)$, where $(\alpha,\beta,d)$ satisfy the conditions (a), (b) and (c). We know that $\delta \equiv 0$ by Proposition \ref{p:gammavanishes}. We have 
\[ \left\{\begin{array}{lllll}
\deg(\UU) & = & -\frac{1}{2}(p\vert\alpha\vert+q\vert\beta\vert -(q-p)a-p) & = & -p(a-1) \\
\deg(\VV) & = & -\frac{1}{2}(p\vert\alpha\vert+q\vert\beta\vert +(q-p)a + p) & = & -qa
\end{array}\right..\] 
Let $\LL\subset \UU$ be a line bundle of maximal degree $l$. The bundle $\LL_\bullet\oplus \VV_\bullet$ is $\gamma\oplus 0$-invariant so we have, by semistability,
\[l + \vert \alpha \vert + \deg(\VV)+q\vert\beta \vert\leq 0~,\]
which gives
\begin{eqnarray*}
l & \leq & qa-(q+1)\vert\beta\vert +\epsilon \\
& = & qa -(q+1)a+\frac{(q+1)p}{p+q}(1-\epsilon)+\epsilon \\
& = & -a+1+\frac{q(p-1)}{p+q}(1-\epsilon). \\
\end{eqnarray*}
If $(p,q)=(2,2)$, or $p=1$ or $q=1$, the condition $\epsilon>0$ gives 
\[l\leq -a+1+\frac{p(q-1)}{p+q}<-a+2.\]
Otherwise, $\epsilon>\frac{2pq-2p-2q}{2pq-p-q}$ and so we get
\begin{eqnarray*}
l & \leq & -a + 1 + \frac{q(p-1)}{p+q}\left(1-\frac{2pq-2p-2q}{2pq-p-q} \right) \\
& = & -a +1+\frac{pq-q}{2pq-p-q} \\
& < & -a+2
\end{eqnarray*}
In particular, $l\leq -a+1$. But on the other hand, by maximality, $l\geq \frac{\deg(\UU)}{p}=-a+1$. So $l=-a+1$ and $\UU=\OO(-a+1)\oplus\cdots\oplus \OO(-a+1)$.
\end{proof}
\begin{rmk} \label{rmk:DecompUVGenericWeights}
In the proofs of Lemmas \ref{l:decompositionU} and \ref{l:decompositionV}, the assumption that $(\alpha,\beta)$ is constant simplifies the notations a lot. However, one could easily extend the proof to $\SU(p,q)$-multiweights in a neighborhood of a constant multiweight satisfying the required hypotheses.
\end{rmk}

For the rest of this section, we fix an integer $a$ and a constant multiweight $(\alpha,\beta)$ stafisfying the conditions of Lemma \ref{l:decompositionV}. We take $\UU = \OO(-a+1)\oplus \cdots \oplus \OO(-a+1) = \C^p \otimes \OO(-a+1)$, $\VV = \OO(-a)\oplus \cdots \oplus \OO(-a) = \C^q\otimes \OO(-a)$ and set $d= (q-p)a + p$, so that $\deg \UU - \deg \VV = d$.

According to Lemmas \ref{l:decompositionU} and \ref{l:decompositionV}, every semistable parabolic $\SU(p,q)$-Higgs bundle in $\MM(\alpha,\beta,d)$ is of the form $(\UU_\bullet\oplus \VV_\bullet, \gamma \oplus 0)$, where the parabolic structure on $\UU_\bullet$ and $\VV_\bullet$ at $x_j$ is given by trivial flags with weights $\alpha^j$ and $\beta^j$ respectively, and for some $\gamma \in H^0(\KK(D)\otimes\Hom(\UU,\VV))$.  \footnote{Note that we do have $\deg(\UU) + \dev(\VV) + \norm{\alpha} + \norm{\beta} = 0$.} We are left with understanding which choices of $\gamma$ actually produce a (semi)stable Higgs bundle.

\smallskip
The identifications $\UU = \C^p \otimes \OO(-a+1)$ and $\VV = \C^q \otimes \OO(-a)$ induce an isomorphism
\begin{eqnarray*}
H^0(\KK(D) \otimes \Hom(\UU,\VV))& \cong & H^0(\Hom(\C^p\otimes \OO(-a+1), \C^q\otimes \OO(-a+s-2))\\
&\cong & \Hom(\C^p,\C^q)\otimes H^0(\OO(s-3))~.
\end{eqnarray*}
Note that $H^0(\OO(s-3))$ has dimension $s-2$. Recall from Section \ref{ss:KroneckerVarieties} that $E(p,q,r)=\Hom(\C^p,C^q)^r$, so fixing a basis of $H^0(\OO(s-3))$, we eventually obtain an isomorphism
\[\left\{\begin{array}{rcl}
E(p,q,s-2) & \longrightarrow & H^0(\KK(D)\otimes \Hom(\UU,\VV)) \\
\mathbf A & \longmapsto & \gamma_{\mathbf{A}} \end{array}\right..\]

\begin{lem} \label{l:EquivalenceStability}
Let $\mathbf A$ be a point in $E(p,q,s-2)$. The following are equivalent:
\begin{itemize}
\item the parabolic $\SU(p,q)$-Higgs bundle $(\UU\oplus \VV, \gamma_{\mathbf A} \oplus 0)$ is semistable (resp. stable), 
\item $\mathbf A$ is a semistable (resp. stable) point of the space $E(p,q,s-2)$ in the sense of Theorem~\ref{t:KingStability}.
\end{itemize}
\end{lem}

\begin{proof}
Let $\UU'$ and $\VV'$ be subbundles of $\UU$ and $\VV$ respectively. Remark first that
\[\deg \UU' \leq (-a+1) \rank (\UU')\quad \textrm{and} \quad \deg \VV' \leq -a \rank (\VV')~,\]
with equalities if and only if $\UU' = U'\otimes \OO(-a+1)$ and $\VV' = V'\otimes \OO(-a)$ for some subspaces $U'$ and $V'$ of $\C^p$ and $\C^q$ respectively. In that case, moreover, we have that 
\[\gamma_{\mathbf A}(\UU')\subset \VV'\]
if and only if 
\[A^j(U')\subset V'\]
for all $j\in \{1,\ldots, s-2\}$.

If $\mathbf A$ is unstable, one can find $U'\subset\C^p$ and $V' \subset \C^q$ such that $A^j(U')\subset V'$ for all $j\in \{1,\ldots, s-2\}$ and such that $\dim V' < \frac{q}{p} \dim U'$. Set $\UU' = U'\otimes \OO(-a+1) \subset \UU$ and $\VV' = V'\otimes \OO(-a) \subset \VV$. Then $\UU'\oplus \VV'$ is invariant by $\gamma_{\mathbf A} \oplus 0$. Using $\vert\alpha\vert=a-\frac{p+q\epsilon}{p+q}$ and $\vert\beta\vert=a-\frac{p-p\epsilon}{p+q}$, we have
\begin{eqnarray*}
\deg(\UU'_\bullet\oplus \VV'_\bullet) & = & - (a-1) \dim U' - a \dim V' + \vert\alpha \vert \dim U' + \vert \beta \vert \dim V'  \\
&=& \dim U' - \frac{\left( (p + q\epsilon)\dim U' + (p- p\epsilon)\dim V'\right)}{p+q} \\
&=& \frac{(q\dim U' - p\dim V')(1-\epsilon)}{p+q}\\
&> & 0 \quad \textrm{since $p \dim V' < q \dim U'$ and $\epsilon < 1$.}
\end{eqnarray*}
Thus $(\UU\oplus \VV, \gamma_{\mathbf A} \oplus 0)$ is unstable.

\smallskip
Conversely, assume that $\mathbf A$ is semistable. Let $\UU'$ and $\VV'$ be two subbundles of $\UU$ and $\VV$ of respective rank $m$ and $n$ such that $\UU' \oplus \VV'$ is invariant by $\gamma\oplus 0$ (\emph{i.e.} such that $\gamma(\VV') \subset \UU'$). Two cases appear:
\medskip
\begin{itemize}
\item If $\deg(\UU') = -(a-1)\rank (\UU')$ and $\deg(\VV) = -a\, \rank (\VV')$, then we have $\UU' = U'\otimes \OO(-a+1)$ as well as $\VV' = V'\otimes \OO(-a)$, where $A^j(U')\subset V'$ for all $j\in \{1,\ldots ,s\}$. By the previous computation we have
\[\deg(\UU'_\bullet \oplus \VV'_\bullet) = \frac{(q \dim U' - p \dim V')(1-\epsilon)}{p+q}~,\]
which is non-positive by semistability of $\mathbf A$.

\medskip
\item Otherwise, $\deg(\UU'\oplus \VV') \leq -(a-1)\rank (\UU') - a\, \rank (\VV') -1$. The same computation as before gives
\begin{equation} \label{eq:stability2}\deg(\UU'_\bullet \oplus \VV'_\bullet) \leq \frac{(q\, \rank (\UU') - p\, \rank (\VV'))(1-\epsilon)}{p+q} -1 \leq \frac{pq(1-\epsilon)}{p+q}-1~.\end{equation}
If $p$ and $q$ are greater than $1$, recall that we imposed the condition
\[\epsilon > \frac{2pq - 2p - 2q}{2pq-p-q} = 1 - \frac{p+q}{2pq-p-q}~.\]
Plugging this in \eqref{eq:stability2}, we get
\[\deg(\UU'_\bullet \oplus \VV'_\bullet) < \frac{pq}{2pq - p -q} - 1 \leq 0\]
since $p+q \leq pq$ for $p$ and $q \geq 2$.

If $p=1$, then we simply use $\epsilon>0$, and \eqref{eq:stability2} gives
\[\deg(\UU'_\bullet \oplus \VV'_\bullet) \leq \frac{-1}{q+1} < 0~.\]
Similarly, if $q=1$, we get
\[\deg(\UU'_\bullet \oplus \VV'_\bullet) \leq \frac{-1}{p+1} < 0~.\\ \]
\end{itemize}

This proves that $(\UU_\bullet \oplus \VV_\bullet , \gamma_{\mathbf A} \oplus 0)$ is semistable if and only if $\mathbf A \in E(p,q,s-2)$ is semistable. Repeating the proofs with strict inequalities instead of inequalities, one obtains the equivalence between the stability of $(\UU_\bullet \oplus \VV_\bullet , \gamma_{\mathbf A} \oplus 0)$ and that of $\mathbf A$.
\end{proof}

\begin{coro}
The moduli space $\MM(\alpha, \beta,d)$ is isomorphic to the Kronecker variety $\RR(p,q,s-2)$. In particular,
\begin{itemize}
\item It is non-empty if and only if $\frac{p}{q}+\frac{q}{p} < s-2$ or $s\in \{3,4\}$ and $p=q$,
\item it contains stable Higgs bundles as soon as $\frac{p}{q}+\frac{q}{p} < s-2$,
\item if $p$ and $q$ are coprime, it contains only stable Higgs bundles.
\end{itemize}
\end{coro}

\begin{proof}

We set once and for all $\UU = \C^p \otimes \OO(-a+1)$ and $\VV = \C^q \otimes \OO(-a)$, and we let $\UU_\bullet$ and $\VV_\bullet$ denote the parabolic bundles with underlying bundles $\UU$ and $\VV$ and parabolic structure at $x_j$ given by trivial flags and respective weights $\alpha^j$ and $\beta^j$. Then the map 
\[\mathbf A \mapsto (\UU_\bullet \oplus \VV_\bullet, \gamma_{\mathbf A}\oplus 0)\]
identifies $E(p,q,s-2)$ with an algebraic family of $\SU(p,q)$-Higgs bundles. Since $\MM(\alpha,\beta)$ is a coarse moduli space of semistable parabolic Higgs bundles and since semistability in the sense of parabolic Higgs bundles and in the sense of Theorem \ref{t:KingStability} are equivalent by Lemma \ref{l:EquivalenceStability}, we obtain an algebraic map $\phi :E^{ss}(p,q,s-2) \to \MM(\alpha,\beta,d)$. By Lemmas \ref{l:decompositionV} and \ref{l:decompositionU}, this map is surjective. Finally, the automorphisms of $\UU = \C^p \otimes \OO(-a+1)$ and $\VV = \C^q \otimes \OO(-a)$ are given respectively by the actions of $\GL(p,\C)$ and $\GL(q,\C)$ on the factors $\C^p$ and $\C^q$. One easily deduces that $(\UU_\bullet\oplus \VV_\bullet, \gamma_{\mathbf A}\oplus 0)$ and $(\UU_\bullet \oplus \VV_\bullet, \gamma_{\mathbf A'}\oplus 0)$ are isomorphic if and only if $\mathbf A$ and $\mathbf A'$ are in the same $G_{p,q}$-orbit. This shows that the map $\phi$ factors through a bijection from $\mathcal{R}(p,q,s-2)$ to $\mathcal{M}(\alpha,\beta,d)$. Since $\mathcal{M}(\alpha,\beta,d)$ is a normal variety (see \cite{Yokogawa2}), this bijection is an isomorphism.
\end{proof}

\subsection{Pertubation of the constant weights}

In this subsection, we remark that the example constructed previously can be perturbed to obtain non-empty compact components for generic multiweights.

\smallskip
Let us start with a fixed integer $a\in\left[\frac{s+q}{p+q},\frac{(p+q-1)s+q}{p+q}\right]$. Set $d = (q-p)a+p$ and choose a constant type $(\alpha^j, \beta^j)$ satisfying the hypotheses of Lemma \ref{l:decompositionV}.

Recall from Section \ref{s:su(p,q)higgsbbundles} that $\WW(s,p,q)$ denotes the space of $\SU(p,q)$-multiweights.

\begin{prop} \label{p:CompactComponentsGeneric}
Assume $s-2 > \frac{p}{q} + \frac{q}{p}$. Then there exists a neighborhood $W(\alpha,\beta)$ of $(\alpha,\beta)$ in $\WW(s,p,q)$ such that for all $(\alpha',\beta')\in W(\alpha,\beta)$, the moduli space $\MM(\alpha',\beta',d)$ is compact and non-empty.
\end{prop}

We can actually be more precise and prove that these compact moduli spaces are isomorphic to some feathered Kronecker varieties introduced in Section \ref{ss:FeatheredKroneckerVarieties}.

\begin{prop} \label{p:FeatheredKroneckerGeneric}
Let $(\alpha',\beta')\in \WW(s,p,q)$ be a multiweight of the form $(\alpha,\beta)+ (\eta,\zeta)$ where
\[\eta_1^j < \cdots < \eta_p^j\]
and
\[\zeta_1^j < \cdots < \zeta_q^j\]
for all $j\in\{1,\ldots, s\}$. Then, for $(\eta,\zeta)$ small enough, $\mathcal M(\alpha',\beta',d)$ is isomorphic to the feathered Kronecker variety $\RR_\spadesuit(p,q,s-2,s,\eta,\zeta)$.
\end{prop}

\begin{rmk}
The assumption that the perturbations of the weights $\eta^j$ and $\zeta^j$ are strictly increasing is only here so that the associated parabolic structures are given by complete flags. As in Section~\ref{ss:FeatheredKroneckerVarieties}, this is merely to avoid a proliferation of indices.
\end{rmk}

\begin{proof}[Proof of Proposition \ref{p:CompactComponentsGeneric}]
The compactness criterion \ref{c-compactnessscriterion} is clearly an open condition. Thus, for $(\alpha',\beta')$ close enough to $(\alpha,\beta)$, $\MM(\alpha',\beta',d)$ is compact. To see that it is non-empty, note that, when $s-2 > \frac{p}{q}+\frac{q}{p}$, there exists a stable parabolic Higgs bundle $(\UU_\bullet\oplus \VV_\bullet, \gamma\oplus 0)$ of type $(\alpha,\beta)$ with $\deg(\UU) - \deg(\VV) = d$. Let $(\UU'_\bullet\oplus \VV'_\bullet, \gamma\oplus 0)$ be a parabolic Higgs bundle with the same underlying holomorphic bundle $\UU\oplus \VV$ and a parabolic structure of type $(\alpha',\beta')$. Since stability is an open condition, $(\UU'_\bullet\oplus \VV'_\bullet, \gamma\oplus 0)$ is stable for $(\alpha',\beta')$ close enough to $(\alpha, \beta)$, proving that $\MM(\alpha',\beta',d)$ is non-empty. 
\end{proof}

\begin{proof}[Proof of Proposition \ref{p:FeatheredKroneckerGeneric}]
First, let us associate to a point in $E_\spadesuit(p,q,s-2,s)$ a parabolic $\SU(p,q)$-Higgs bundle of type $(\alpha',\beta')$. To do so, let us fix a basis of $H^0(\OO(s-3))$ so as to identify $E(p,q,s-2)$ with $\Hom(\C^p, \C^q \otimes H^0(\OO(s-3)))$.

Set $\UU= \C^p\otimes \OO(-a+1)$ and $\VV = \C^q \otimes \OO(-a)$. The canonical bases of $\C^p$ and $\C^q$ define bases of $\UU_x$ and $\VV_x$ at any point $x$, up to scaling. This induces canonical isomorphisms between the space of complete flags in $\C^p$ (resp. $\C^q$) and the space of complete flags on $\UU_x$ (resp. $\VV_x$). Given a point $(\mathbf F,\mathbf H)\in \FF^s \times \HH^s$, we now see the flag $F^j$ (resp. $H^j$) as a flag on $\UU_{x_j}$ (resp. $\VV_{x_j}$). We denote by $\UU_\bullet(\mathbf F, \alpha')$ and $\VV_\bullet(\mathbf H, \beta')$ the parabolic bundles with underlying bundles $\UU$ and $\VV$ respectively, and parabolic structure at $x_j$ given by the complete flags $F^j$ and $H^j$ and the weights ${\alpha'}^j$ and ${\beta'}^j$. Let us also denote respectively by $\UU_\bullet(\alpha)$ and $\VV_\bullet(\beta)$ the parabolic bundles with underlying bundle $\UU$ and $\VV$ and parabolic structure at $x_j$ given by the trivial flags and the constant weights $\alpha^j$ and $\beta^j$.

Finally, recall that we can associate to $\mathbf A \in E(p,q,s-2) \cong  \Hom(\C^p, \C^q) \otimes H^0(\OO(s-3))$ an element $\gamma_{\mathbf A} \in \Hh^0(\KK(D)\otimes \Hom(\UU,\VV))$.
We can thus associate to a point $(\mathbf A,\mathbf F,\mathbf H)\in E_\spadesuit(p,q,s-2,s)$ the parabolic $\SU(p,q)$-Higgs bundle 
\[\big(\UU_\bullet(\mathbf F,\alpha') \oplus \VV_\bullet(\mathbf H, \beta'), \gamma_{\mathbf A}\oplus 0\big),\]
identifying $E_\spadesuit(p,q,s-2,s)$ to an algebraic family of parabolic $\SU(p,q)$-Higgs bundles of type $(\alpha',\beta')$.\footnote{Importantly, since $\alpha^j < \beta^j$ for all $j$, then, for $\eta$ and $\zeta$ small enough, we still have ${\alpha'}_p^j < {\beta'}_1^j$, from which we obtain that $\KK(D)\otimes \Hom(\UU_\bullet(\mathbf F, \alpha') ,\VV_\bullet(\mathbf H, \beta')) = \KK(D) \otimes \Hom(\UU,\VV)$.}

Let $(\mathbf A, \mathbf F,\mathbf H)$ be a point in $E_\spadesuit(p,q,s-2,s)$. Let $\UU'$ and $\VV'$ be subbundles of $\UU$ and $\VV$ such that $\gamma_{\mathbf A}(\UU')\subset \VV'$. Denote respectively by $\UU'_\bullet(\alpha)$, $\UU'_\bullet(\mathbf F,\alpha')$, $\VV'_\bullet(\beta)$ and $\VV'_\bullet(\mathbf H,\beta')$ the induced parabolic subbundles of $\UU_\bullet(\alpha)$, $\UU_\bullet(\mathbf F,\alpha')$, $\VV_\bullet(\beta)$ and $\VV_\bullet(\mathbf H,\beta')$. Repeating the computations of the proof of Lemma \ref{l:EquivalenceStability}, we see that:
\begin{itemize}
\item[(1)] Either $\UU' = U'\otimes \OO(-a+1)$ and $\VV' = V'\otimes \OO(-a)$ where $U$ and $V$ are respectively subspaces of $\C^p$ and $\C^q$ such that $A^j(U')\subset V'$ for all $j\in \{1,\ldots, s-2\}$,
\item[(2)] Or $\deg(\UU'_\bullet(\alpha)) + \deg (\VV'_\bullet(\beta)) <0$ and thus $\deg(\UU'_\bullet(\mathbf F, \alpha')) + \deg (\VV'_\bullet(\mathbf H, \beta')) <0$ for $(\alpha',\beta')$ close enough\footnote{One can easily verify that everytime we ask for $(\alpha',\beta')$ ``close enough'' to $(\alpha,\beta)$, there actually is an explicit such choice which does not depend on $(\mathbf A,\mathbf F,\mathbf H)$, so that in the end we are indeed describing a neighborhood of $(\alpha,\beta)$.} to $(\alpha,\beta)$.
\end{itemize}

In case (1), we have the following :
\begin{itemize}
\item Either $\dim V' < \frac{q}{p} \dim U'$, in which case we have $\deg(\UU'_\bullet(\alpha)) + \deg (\VV'_\bullet(\beta)) >0$ and thus $\deg(\UU'_\bullet(\mathbf F, \alpha')) + \deg (\VV'_\bullet(\mathbf H, \beta')) >0$ for $(\alpha',\beta')$ close enough to $(\alpha,\beta)$,

\item Or $\dim V' > \frac{q}{p} \dim U'$, in which case we have $\deg(\UU'_\bullet(\alpha)) + \deg (\VV'_\bullet(\beta)) <0$ and the inequality $\deg(\UU'_\bullet(\mathbf F, \alpha')) + \deg (\VV'_\bullet(\mathbf H, \beta')) <0$ holds for $(\alpha',\beta')$ close enough to $(\alpha,\beta)$,

\item Or $\dim V' = \frac{q}{p} \dim U'$, in which case we have $\deg(\UU'_\bullet(\alpha)) + \deg (\VV'_\bullet(\beta)) = 0$ and
\begin{eqnarray*}
\deg(\UU'_\bullet(\mathbf F, \alpha')) + \deg (\VV'_\bullet(\mathbf H, \beta')) & = &\deg(\UU'_\bullet(\alpha)) + \deg (\VV'_\bullet(\beta)) + \vert \eta(U'\cap \mathbf F)\vert + \vert \zeta(V'\mathbf H)\vert \\
&=& \vert \eta(U'\cap \mathbf F)\vert + \vert\zeta(V'\cap \mathbf H)\vert ~.
\end{eqnarray*}
\end{itemize}
Comparing to Corollary \ref{c:FeatherStabilitySmallEtaZeta} and using $\norm{\eta} + \norm{\zeta} = 0$, we conclude that, for $(\eta,\zeta)$ small enough, the Higgs bundle $(\UU_\bullet(\mathbf F,\alpha')\oplus \VV_\bullet(\mathbf H,\beta'), \gamma_{\mathbf A} \oplus 0)$ is semistable if and only if $(\mathbf A,\mathbf F,\mathbf H)$ is $(\eta,\zeta)$-semistable. Since $\MM(\alpha',\beta')$ is a coarse moduli space, we thus have a morphism $\phi: E^{ss}_\spadesuit(p,q,s-2,s,\eta,\zeta)\to \MM(\alpha',\beta',d)$ for $(\alpha',\beta')$ close enough to $(\alpha,\beta)$.

Again, one easily verifies that two points in $E_\spadesuit(p,q,s-2,s)$ are associated to isomorphic parabolic $\SU(p,q)$-Higgs bundles if and only if they are in the same $G_{p,q}$-orbit. Therefore, the morphism $\phi$ factors through an injective morphism $\bar{\phi}: \RR(p,q,s-2,s, \eta,\zeta) \to \MM(\alpha',\beta',d)$. 

Finally, repeating the proofs of Lemmas \ref{l:decompositionU} and \ref{l:decompositionV} (see Remark \ref{rmk:DecompUVGenericWeights}), one obtains that, for $(\alpha',\beta')$ close enough to $(\alpha,\beta)$, every semistable $\SU(p,q)$-Higgs bundle in $\MM(\alpha,\beta,d)$ has underlying vector bundle $\OO(-a+1) \otimes \C^p \oplus \OO(-a)\otimes \C^q$. It follows that $\bar{\phi}$ is surjective.
\end{proof}

\section{Properties of our compact relative components}\label{s:PropertiesOfOurReps}

In this section we return to the ``representation variety'' side of the non-Abelian Hodge correspondence where we summarize the results obtained so far and prove Theorems \ref{t:ExistenceCompactComponents} and \ref{t-thm2} for $\SU(p,q)$, under the condition that $\frac{p}{q}+ \frac{q}{p} < s-2$.

\subsection{An open set of compact components} \label{ss:OpenSetCompactCpts}

Recall that, given a $\SU(p,q)$-multiweight $(\alpha,\beta)\in\mathcal{W}(s,p,q)$ (see Section \ref{s:su(p,q)higgsbbundles}), we denote by $h(\alpha,\beta)$ the tuple of matrices $(h_1,\ldots, h_s)$, where $h_j$ is diagonal with eigvenvalues
\[e^{2\pi i \alpha_1^j},\ldots, e^{2\pi i \alpha_p^j}, e^{2\pi i \beta_1^j}, \ldots, e^{2\pi i \beta_q^j}~.\]
Recall also from Subsection \ref{ss:NAH} that the non-Abelian Hodge correspondence gives a homeomorphism
\[\mathsf{NAH}: \MM(\alpha,\beta,d) \overset{\cong}{\to} \Rep_{h(\alpha,\beta)}^{\Vert\beta\Vert-\Vert\alpha\Vert-d}(\Sigma_{0,s}, \SU(p,q))~.\\ \]

Now, let us fix an integer $a\in\left[\frac{s+p}{p+q},\frac{(p+q-1)s+p}{p+q}\right]$ and let $(\alpha,\beta)$ be a constant multi-weight satisfying the hypotheses of Lemma \ref{l:decompositionV}. Through the non-Abelian Hodge correspondence, the results of the previous section translate into the following theorem:

\begin{theo}
Assume $s>2+ \frac{p}{q}+\frac{q}{p}$. Then there exists a neighborhood $W(\alpha,\beta)$ of $(\alpha,\beta)$ in $\WW(s,p,q)$ such that, for all $(\alpha',\beta')\in W(\alpha,\beta)$, the relative component
\[\Rep_{h(\alpha',\beta')}^{\Vert\beta'\Vert-\Vert\alpha'\Vert-d}\big(\Sigma_{0,s},\SU(p,q)\big)\]
is compact, non-empty, and contains an irreducible representation.
\end{theo}

It would be interesting to replace ``irreducible'' by ``Zariski dense'' in the above theorem. Though we strongly suspect that each of our compact components contains a Zariski dense representation, we do not know of a quick argument to rule out all other possible Zariski closures in general. We can nevertheless make the improvement for ``almost all'' weights $(\alpha',\beta')$, thanks to the following remark:

\begin{lem} \label{l:OpenSet}
The union 
\[\Omega = \bigsqcup_{(\alpha',\beta')\in W(\alpha,\beta)} \Rep_{h(\alpha',\beta')}^{\Vert\beta'\Vert-\Vert\alpha'\Vert-d}\big(\Sigma_{0,s},\SU(p,q)\big)\]
has non-empty interior in the (absolute) character variety $\Rep\big(\Sigma_{0,s},\SU(p,q)\big)$.
\end{lem}

\begin{proof}
Let $\rho_0$ be a representation in $\Rep_{h(\alpha,\beta)}^{\Vert\beta\Vert-\Vert\alpha\Vert-d}(\Sigma_{0,s},\SU(p,q))$. We simply need to remark that, for every representation $\rho: \Gamma_{0,s}\to \SU(p,q)$ in a neighborhood of $\rho_0$ and every $j\in \{1,\ldots, s\}$, $\rho(c_j)$ has all its eigenvalues of modulus $1$. Indeed, $\rho_0(c_j)$ has two eigenspaces of respective dimension $p$ and $q$ in restriction to which the pseudo-Hermitian form $\mathbf h$ of $\C^{p,q}$ is respectively positive definite and negative definite. In particular, the eigendirections of $\rho_0(c_j)$ are ``far'' from the isotropic cone of equation $\mathbf h = 0$. Thus, for $\rho$ close enough to $\rho_0$, $\rho(c_j)$ does not have any isotropic eigenvector either. Since $\rho(c_j)$ preserves $\mathbf h$, every non-isotropic eigenvector has a corresponding eigvenvalue of modulus $1$.

In conclusion, every $\rho$ in a neighborhood of $\rho_0$ belongs to $\Rep_{h(\alpha',\beta')}(\Sigma_{0,s},\SU(p,q))$ for some $(\alpha',\beta')$ in a neighborhood of $(\alpha,\beta)$. Moreover, the Toledo invariant of $\rho$ has the form $d'+\norm{\alpha'}- \norm{\beta'}$ for some $d'\in \Z$. By continuity of the Toledo invariant, $d' = d$ for $\rho$ close enough to $\rho_0$. Hence $\rho$ belongs to $\Rep_{h(\alpha',\beta')}^{\Vert\beta'\Vert-\Vert\alpha'\Vert-d}\big(\Sigma_{0,s},\SU(p,q)\big)$.
\end{proof}

Now, the set of Zariski dense representations has full measure in $\Rep(\Gamma_{0,s}, \SU(p,q))$ (see \cite{ZariskiDenseRepresentations}). We thus obtain the corollary:

\begin{coro}
There is a set $W'$ of full measure in $W(\alpha,\beta)$ such that, for all $(\alpha',\beta')\in W'$, the relative component
\[\Rep_{h(\alpha',\beta')}^{\Vert\beta'\Vert-\Vert\alpha'\Vert-d}(\Gamma_{0,s},\SU(p,q))\]
contains a Zariski dense representation.
\end{coro} 
 This concludes the proof of Theorem \ref{t:ExistenceCompactComponents} for $\SU(p,q)$, under the assumption that $\frac{p}{q}+\frac{q}{p} < s-2$.

\subsection{Proof of Theorem \ref{t-thm2}} \label{ss:ProofTheorem2}
Let us now describe some properties of the representations contained in these compact components and prove Theorem \ref{t-thm2}. Assume $s>2+\frac{p}{q}+\frac{q}{p}$, and let $(\alpha,\beta),~W(\alpha,\beta)$ and $\Omega$ be as in the previous subsection.

\begin{theo} \label{t:PropertiesRepresentationsPrecise}
Every representation $\rho\in \Omega$ has the following properties:
\begin{itemize}
\item[$(i)$] For every identification of $\Sigma_{0,s}$ with a punctured Riemann sphere, there is a $\rho$-equivariant holomorphic map from $\widetilde{\Sigma}_{0,s}$ to the symmetric space of $\SU(p,q)$.
\item[$(ii)$] For every homotopy class of simple closed curve $c$ in $\Sigma_{0,s}$,
the complex eigenvalues of $\rho(c)$ have modulus $1$.
\item[$(iii)$] More generally, for every $k\geq 0$, there is a constant $C(k)$ such that, if $c$ is the homotopy class of a closed curve with at most $k$ self-intersections, then the eigenvalues of $\rho(c)$ have modulus less than $C(k)$.
\end{itemize}
\end{theo}

We first describe the geometry of the symmetric space associated to $\SU(p,q)$. Recall that $\C^{p,q}$ is the usual $(p+q)$-dimensional complex vector space equipped with a signature $(p,q)$ Hermitian product $\h$. Let $\Y$ be the Grassmanian of $p$-dimensional (complex) linear subspaces of $\C^{p,q}$ on which the restriction of $\h$ is positive definite. Given $P\in\Y$, the stabilizer of $P$ in $\SU(p,q)$ is a maximal compact subgroup conjugated to $\K:=\textrm{S}(\text{U}(p)\times \text{U}(q))$, identifying $\Y$ with the symmetric space $\SU(p,q)/\K$. 

The Grassmanian $\Y$ carries two natural tautological bundles. Consider the trivial bundle 
\[E:=\Y\times\C^{p,q}\to \Y\]
and define the rank $p$ vector subbundle $U\subset E$ by $U_P=P$. Similarly, one has the subbundle $V\subset E$ defined by $V_P=P^\bot$. In particular, $E=U\oplus V$.

The tangent space $T_P\Y$ is identified with $\Hom_\C(P,P^\bot)$, so the tangent bundle $T\Y$ is isomorphic to the bundle $\Hom_\C(U,V)$. Its complexification is thus given by
\[T^\C\Y = T^{1,0}\Y \oplus T^{0,1}\Y = \Hom_\C(U,V)\oplus \Hom_\C(V,U),\]
where $T^{1,0}\Y$ and $T^{0,1}\Y$ are respectively the holomorphic and anti-holomorphic tangent bundle of $\Y$.

\smallskip
The non-Abelian Hodge correspondence relies on the existence, for any irreducible representation $\rho: \Gamma \to \SU(p,q)$ and Riemann surface structure $X$ on $\Sigma$, of a unique $\rho$-equivariant harmonic map $f: \widetilde{X} \to \Y$. More precisely, the $\SU(p,q)$-Higgs bundle $(\UU\oplus\VV,\gamma\oplus\delta)$ on $X$ associated to $\rho$ is 
\[\left\{ \begin{array}{lll}
\UU\oplus\VV & = & f^*(U\oplus V) \\
\gamma\oplus\beta & = & \partial f \end{array}\right.,\]
where $\partial f\in \Omega^{1,0}\big(X,\Hom_\C(\UU,\VV)\oplus\Hom_\C(\VV,\UU)\big)$ is the $(1,0)$-part of the complexification of $df$.

\smallskip
A map $f:\widetilde X \to \Y$ is holomorphic if and only if $\partial f$ takes value in $T^{1,0}\Y$. It follows that a $\SU(p,q)$-Higgs bundle $(\UU\oplus\VV,\gamma\oplus\delta)$ comes from an equivariant holomorphic map $f:\widetilde X \to \Y$ if and only if $\delta\equiv 0$. Since this is the case for all parabolic Higgs bundles satisfying our compactness criterion (Corollary \ref{c-compactnessscriterion}), Property $(i)$ follows.

\smallskip
Properties $(ii)$ and $(iii)$ follow from $(i)$ using the contraction property of holomorphic maps. We first introduce the Kobayashi metric on $\Y$. By a theorem of Harish-Chandra, $\Y$ is biholomorphic to a bounded domain in $\C^n$, and the Kobayashi distance $d_K$ between two points $x,y\in\Y$ can be defined as
\[d_K(x,y) =\inf\big\{d_P(a,b),~a,b\in \D,~g(a)=x,~g(b)=y \text{ for some } g:\D\to\Y \text{ holomorphic}  \big\}. \]
Here, $\D$ is the unit disk in $\C$ equipped with the Poincar\'e distance $d_P$.

The Kobayashi metric is a $\SU(p,q)$ invariant Finsler metric on $\Y$ which is thus bi-Lipschitz to the symmetric Riemannian metric $d_\Y$. Moreover, by definition, any holomorphic map $f: (\D,d_P) \to (\Y,d_K)$ is contracting.

\smallskip
Given an element $g\in\SU(p,q)$, define the \emph{translation length} of $g$ by 
\[t(g) = \inf\left\{d_\Y(x,gx),~ x\in\Y \right\},\]
where $d_\Y$ is the distance associated to the Killing metric. Up to a multiplicative constant, we have
\[t(g) = \sqrt{\sum_{i=1}^{p+q}(\log\vert\lambda_i\vert)^2},\]
where the $\lambda_i$ are the eigenvalues of $g$.

\smallskip
We are now ready to prove $(ii)$ and $(iii)$. Take $\rho: \Gamma_{0,s} \to \SU(p,q)$ a representation in $\Omega$ and let $\gamma\in\Gamma_{0,s}$ be a homotopy class of simple closed curve on $\Sigma_{0,s}$. For any complex structure $X$ on $\Sigma_{0,s}$, Property $(i)$ gives the existence of a $\rho$-equivariant holomorphic map $f: \widetilde X \to \Y$. Given a biholomorphism $\widetilde X \cong \D$, and $x\in\D$ a point on (a lift of) the geodesic homotopic to $\gamma$, we deduce
\begin{eqnarray*}
t(\rho(\gamma)) & \leq & d_\Y\big(f(x),f(\gamma x)\big) \\ 
& = & d_\Y\big(f(x),\rho(\gamma)f(x)\big) \\
& \leq & C d_K \big(f(x),\rho(\gamma)f(x)\big) \\
& \leq & C d_P(x,\gamma x)=Cl_X(\gamma)~,
\end{eqnarray*}
where $l_X(\gamma)$ is the hyperbolic length of the geodesic representing $\gamma$ and $C$ is a constant.

This inequality being true for any complex structure on $\Sigma_{0,s}$, we obtain that, up to a multiplicative constant, the translation length of $\rho(\gamma)$ is less than the infimum of the hyperbolic length of $\gamma$ over the Teichm\"uller space of $\Sigma_{0,s}$.

It is well-known that this infimum is $0$ when $\gamma$ is simple, proving $(ii)$. When $\gamma$ has $k$ self-intersections, this infimum is less than a constant depending only on $k$, implying $(iii)$. 

\smallskip
This concludes the proof of Theorem \ref{t:PropertiesRepresentationsPrecise} and that of Theorem \ref{t-thm2} for $\SU(p,q)$ with the assumption $\frac{p}{q}+ \frac{q}{p} < s-2$.

\section{More compact components}\label{s:MoreComponents}

In this section, we give ways of constructing more compact components : by tensorizing with a line bundle, by restricting to representations into other Hermitian Lie groups, and by restricting to subsurfaces. These examples show how far we are from a classification of compact connected components.

\subsection{Tensor product with a line bundle}

There is a notion of projective equivalence for parabolic $\SU(p,q)$-Higgs bundles.

 Under the non-Abelian Hodge correspondence, projectively equivalent polystable $\SU(p,q)$-Higgs bundles correspond to representations $\rho: \Gamma \to \SU(p,q)$ with the same image in $\PU(p,q)=\SU(p,q)/Z$, where \[Z=\left\{e^{\frac{2ik\pi}{p+q}}\Id,~k=0,\ldots,p+q-1 \right\}\cong \Z_{p+q}\] is the center of $\SU(p,q)$. In particular, given a representation $\rho: \Gamma \to \SU(p,q)$, the set of representations $\rho': \Gamma \to \SU(p,q)$ with the same image in $\PU(p,q)$ is parametrized by $H^1(\Sigma_{0,s},\Z_{p+q})$.

\begin{defi}
Two parabolic Higgs bundles $(\UU_\bullet\oplus \VV_\bullet,\gamma\oplus\delta)$ and $(\UU'_\bullet\oplus \VV'_\bullet,\gamma'\oplus\delta')$ are \textit{projectively equivalent} if there is a parabolic line bundle $\LL_\bullet$ such that $\UU'_\bullet=\LL_\bullet\otimes \UU_\bullet,~\VV'_\bullet=\LL_\bullet\otimes \VV_\bullet$ and $(\gamma',\delta')=(\gamma,\delta)$.
\end{defi}

The condition $\det(\UU'_\bullet)\otimes\det(\VV'_\bullet)=\det(\UU_\bullet)\otimes\det(\VV_\bullet)=\OO$ implies that $(\LL_\bullet)^{p+q}=\OO$. That is, two projectively equivalent parabolic $\SU(p,q)$-Higgs bundles differ by a (parabolic) $(p+q)$\textsuperscript{th} root of the trivial bundle. If we denote by $\mathcal{J}_X$ the group of degree $0$ parabolic line bundles over $X$, the set of $(p+q)$\textsuperscript{th} roots of the trivial bundle is the subgroup of $(p+q)$ torsion points and is denoted $\mathcal{J}_X[p+q]$.

\begin{lem}
The group $\mathcal{J}_X[p+q]$ of parabolic line bundles $\LL_\bullet$ on $X$ such that $(\LL_\bullet)^{p+q}=\OO$ is isomorphic to $H^1(\Sigma_{0,s},\Z_{p+q})$.
\end{lem} 

\begin{proof}
Consider the presentation $\Gamma_{0,s}=\langle c_1,\ldots,c_s \vert c_1\cdots c_s=1\rangle$ of the fundamental group $\Gamma_{0,s}$ of $\Sigma_{0,s}$. An element $\varphi\in H^1(\Sigma_{0,s},\Z_{p+q})$ is thus a map $\varphi: \{c_1,\ldots,c_s\} \to \Z_{p+q}$ such that $\sum_{i=1}^s\varphi(c_i) = 0\in \Z_{p+q}$.

The map $\varphi$ has a unique representative $\hat\varphi : \{c_1,\ldots,c_s\} \to \{0,\ldots,p+q-1\}$ such that $\vert\hat\varphi\vert = \sum_{i=1}^s \hat\varphi(c_i) \equiv 0 \mod[p+q]$. Define the corresponding parabolic line bundle $\LL_\bullet^\varphi$ by
\[\LL_\bullet^\varphi = \OO\left(-\frac{\vert\hat\varphi\vert}{p+q} + \sum_{i=1}^s \frac{\hat\varphi(c_i)}{p+q}x_i\right).\]
One easily checks that $(\LL^\phi_\bullet)^{p+q}=\OO$ and that the map $\varphi \mapsto \LL_\bullet^\varphi$ is a group morphism.

Let us show that the map is surjective. Given $\LL_\bullet= \OO\big(l+\sum_{i=1}^s\alpha^ix_i \big)$ a $(p+q)^{th}$ root of $\OO$, the condition $(\LL_\bullet)^{p+q}=\OO$ implies that $\alpha^i=\frac{k_i}{p+q}$ for some $k_i\in\N$ and $l=-\frac{\vert k \vert}{p+q}\in\Z$. In particular, there is a unique $\varphi\in H^1\big(\Sigma_{0,s},\Z_{p+q}\big)$ with $\hat\varphi(c_i)=k_i$.
\end{proof}

We get the following corollary:

\begin{coro}
If $(\alpha,\beta)$ and $(\alpha',\beta')$ are two $\SU(p,q)$-multiweights such that, at any point $x_i$, we have 
\[{(\alpha'}^i_1-\alpha^i_1),~ \ldots~,~ ({\alpha'}^i_p- \alpha^i_p),~({\beta'}^i_1-\beta^i_1),~ \ldots~,~ ({\beta'}^i_q - \beta^i_q) \in\frac{1}{p+q}\Z~,\]
then there exists $\LL_\bullet\in \mathcal{J}_X[p+q]$ such that $\LL_\bullet\otimes-:~\MM(\alpha,\beta) \longrightarrow \MM(\alpha',\beta')$ is an isomorphism.
\end{coro}

As a result, given a compact conponent $\MM(\alpha,\beta,d)$, one gets a family of compact components 
\[\big\{\MM(\alpha',\beta',d')=\LL_\bullet\otimes\MM(\alpha,\beta,d),~\LL_\bullet\in \mathcal{J}_X[p+q] \big\}.\] These components correspond to different lifts of the same representations in $\PU(p,q)$.

Note that the condition $\alpha_i^1<\beta_i^q$ we imposed when looking for compact components is generally not preserved by taking the tensor product with $\LL_\bullet\in \mathcal{J}_X[p+q]$. This shows that the conditions of our compactness criterion (Corollary \ref{c-compactnessscriterion}) are too restrictive to account for all compact components.

\subsection{Other Hermitian Lie groups} \label{ss:MoreLieGroups}

We extend here our construction to other classical Hermitian Lie groups, namely $\Sp(2p,\R)$ and $\SO^*(2p)$.

\subsubsection{Compact components in $\Sp(2p,\R)$} Let $(V,\omega)$ be a real vector space of dimension $2p$ together with a symplectic form $\omega$. The group $\Sp(2p,\R)$ is the group of linear transformations of $V$ preserving $\omega$. 

The group $\Sp(2p,\R)$ acts on the complexification $V^\C$ of $V$ which is naturally equipped with the Hermitian form $\h$ defined by $\h(u,v)=i\omega^\C(u,\overline{v})$, where $\omega^\C$ is the $\C$-linear extension of $\omega$, and one can check that $\h$ has signature $(p,p)$. In particular, $(V^\C,\h)\cong \C^{p,p}$ and we get an embedding $\Sp(2p,\R)\hookrightarrow \SU(p,p)$.

Via this embedding, the relative $\Sp(2p,\R)$ character varieties are identified with closed subsets of the $\SU(p,p)$ character varieties. We will prove the following lemma:

\begin{lem} \label{l:RestrictionToSP(2p,R)}
Assume $s\geq 5$ is odd. Then we can find a subset $\Omega \subset \Rep(\Sigma_{0,s},\SU(p,p))$ as in Section~\ref{s:PropertiesOfOurReps} which contains a $\Sp(2p,\R)$ representation. 
\end{lem}

Let us explain first why the lemma implies Theorems \ref{t:ExistenceCompactComponents} and \ref{t-thm2} for $\Rep(\Sigma_{0,s},\Sp(2p,\R))$ with $s$ odd. Indeed, the intersection of $\Omega \cap \Rep(\Sigma_{0,s},\Sp(2p,\R))$ is an open subset of the $\Sp(2p,\R)$ character variety which thus contains a Zariski dense representation. Moreover, the intersection of $\Rep(\Sigma_{0,s},\Sp(2p,\R))$ with a compact component of a relative $\SU(p,p)$ character variety is a compact component of a relative $\Sp(2p,\R)$ character variety. Hence $\Omega \cap \Rep(\Sigma_{0,s},\Sp(2p,\R))$ is a union of compact relative components, proving Theorem \ref{t:ExistenceCompactComponents}. Finally, properties of Theorem \ref{t-thm2} are satisfied by all the representations in $\Omega$, including those in $\Rep(\Sigma_{0,s},\Sp(2p,\R))$ (for the third property, it is important that the embedding $\Sp(2p,\R) \hookrightarrow \SU(p,p)$ induces a holomorphic embedding of the symmetric spaces).

\begin{proof}[Proof of Lemma \ref{l:RestrictionToSP(2p,R)}]

We transit once again through the non-Abelian Hodge correspondence. One can indeed detect which parabolic $\SU(p,p)$-Higgs bundles give rise to a representation in $\Sp(2p,\R)$.

A parabolic $\Sp(2p,\R)$-Higgs bundle is a $\SU(p,p)$-Higgs bundle $(\UU_\bullet\oplus \VV_\bullet,\gamma\oplus\delta)$ such that 
\begin{itemize}
\item $\VV_\bullet=\UU_\bullet^\vee= \Hom(\UU_\bullet,\C)$, 
\item $\gamma \in H^0(\KK(D) \otimes \Hom(\UU_\bullet, \UU_\bullet^\vee))$ and $\delta \in H^0(\KK(D) \otimes \Hom(\UU_\bullet^\vee, \UU_\bullet))$ are self-adjoint.
\end{itemize}
In particular, if $\UU_\bullet$ has parabolic type $\alpha_k=(\alpha^k_1,\ldots,\alpha^k_p)$ at $x_k$, then $\VV_\bullet= \UU_\bullet^\vee$ has type $\beta^k=(\beta^k_1,\ldots,\beta^k_p)$ with $\beta^k_i=1-\alpha^k_{p+1-i}$. It follows that $\Vert\alpha\Vert+\Vert\beta\Vert=ps$.

\smallskip
Let us give explicitly an integer $a$ and a constant $\SU(p,p)$-multiweight $(\alpha,\beta)$ satisfying the hypotheses of Lemma \ref{l:decompositionV} such that $\MM(\alpha,\beta,d)$ contains a $\Sp(2p,\R)$-Higgs bundle.

Take $a = \frac{s+1}{2}$, and set $\alpha^j = \frac{1-\epsilon^j}{2}$ and $\beta^j = \frac{1+\epsilon^j}{2}$ where the $\epsilon^j \in (0,1)$ satisfy
\[\epsilon \equaldef \sum_{j=1}^s \epsilon^j \in \left(1-\frac{1}{p-1},1\right)~.\]
One easily verifies that these multiweights satisfy the hypotheses of Lemma \ref{l:decompositionV} (with $p=q$). 

Let $\LL_\bullet$ be the parabolic line bundle $\OO(-\frac{s-1}{2} +\sum_{j=1}^s \alpha^j x_j)$. Then $\LL_\bullet^\vee = \OO(-\frac{s+1}{2} +\sum_{j=1}^s \beta^j x_j)$. By Lemmas \ref{l:decompositionU} and \ref{l:decompositionV}, every parabolic $\SU(p,p)$-Higgs bundle in $\MM(\alpha,\beta,p)$ has the form $(\UU_\bullet \oplus \VV_\bullet, \gamma\oplus 0)$, where $\UU_\bullet = \LL_\bullet \otimes \C^p$ and $\VV_\bullet = \LL_\bullet^\vee \otimes \C^p$.

The standard bilinear pairing of $\C^p$ given by
\[(\mathbf z,\mathbf w) \mapsto z_1 w_1 + \cdots +z_pw_p\]
induces an isomorphism between $\VV_\bullet$ and $\UU_\bullet^\vee$. Finally, if $\mathbf A = (A_1,\ldots, A_{s-2})$ is a tuple of endomorphisms of $\C^p$ which are symmetric (with respect to the standard bilinear pairing), then the associated section 
$\gamma_{\mathbf A} \in H^0(\KK(D)\otimes \Hom(\UU_\bullet,\UU^\vee_\bullet))$ is self-adjoint. If, moreover, one of the $A_i$ is invertible, the parabolic bundle $(\UU_\bullet \oplus \VV_\bullet, \gamma_{\mathbf A} \oplus 0)$ is a semistable $\Sp(2p,\R)$-Higgs bundle, giving a $\Sp(2p,\R)$-point in $\MM(\alpha,\beta,p)$.

Let $\Omega$ be the subset of the $\SU(p,p)$ character variety given by
\[\Omega = \bigsqcup_{(\alpha',\beta')\in W(\alpha,\beta)} \Rep_{h(\alpha',\beta')}^p(\Sigma_{0,s},\SU(p,p))\]
where $W(\alpha,\beta)$ is an open neighborhood of $(\alpha,\beta)$ in $\WW(p,p,s)$. Then $\Omega$ contains a representation in $\Sp(2p,\R)$.
\end{proof}

Note that, for the above choice of $(\alpha,\beta)$, the relative component $\Rep_{h(\alpha,\beta)}^p(\Sigma_{0,s}, \Sp(2p,\R))$ could be described as a GIT quotient of the space $\mathcal S(p,s-2)$ of $s-2$-tuples of symmetric matrices of size $p$ by the action of $\GL(p,\C)$ given by 
\[g\cdot \mathbf S = (gS_1 g^T, \ldots, g S_{s-2} g^T)~.\]
One could also define a ``feathered version'' of this space to describe the nearby relative components. We don't know if these spaces have been studied in the litterature.

\subsubsection{Compact components in $\SO^*(2p)$} Recall that $\C^{p,p}$ is the complex vector space $\C^{2p}$ equipped with the signature $(p,p)$ Hermitian form $\h$ defined by
\[\h(z,w)= z_1\overline{w}_1+\cdots+z_p\overline{w}_p -z_{p+1}\overline{w}_{p+1}-\cdots-z_{2p}\overline{w}_{2p}.\]
Consider the $\C$-bilinear quadratic form $\q$ on $\C^{2q}$ defined by $\q(z)=\sum_{k=1}^p z_kz_{k+p}$. The group $\SO^*(2p)$ is the group of linear transofrmations of $\C^{2p}$ preserving both $\h$ and $\q$. In particular, $\SO^*(2p)\subset \SU(p,p)$.

Similarly to the $\Sp(2p,\R)$ case, one can see the $\SO^*(2p)$ character variety as a closed subset of the $\SU(p,p)$ character variety and Theorems \ref{t:ExistenceCompactComponents} and \ref{t-thm2} for $\SO^*(2p)$ and $s$ odd will follow from the following lemma:

\begin{lem} \label{l:RestrictionToSO*(2p)}
Assume $s\geq 5$ is odd. Then we can find a subset $\Omega \subset \Rep(\Sigma_{0,s},\SU(p,p))$ as in Section~\ref{s:PropertiesOfOurReps} which contains a $\SO^*(2p)$ representation. 
\end{lem}

\begin{proof}
The proof is very similar to the $\Sp(2p,\R)$ case and we only highlight the main differences between the two cases.

\smallskip
A parabolic $\SO^*(2p)$-Higgs bundle is a $\SU(p,p)$-Higgs bundle $(\UU_\bullet\oplus \VV_\bullet,\gamma\oplus\delta)$ such that 
\begin{itemize}
\item $\VV_\bullet=\UU_\bullet^\vee= \Hom(\UU_\bullet,\C)$, 
\item $\gamma \in H^0(\KK(D) \otimes \Hom(\UU_\bullet, \UU_\bullet^\vee))$ and $\delta \in H^0(\KK(D) \otimes \Hom(\UU_\bullet^\vee, \UU_\bullet))$ are \emph{antiself-adjoint}.
\end{itemize}

If we make the same choices of $(\alpha,\beta)$ as in the proof of Lemma \ref{l:RestrictionToSP(2p,R)}, $\SO^*(2p)$-Higgs bundles in $\MM(\alpha,\beta,p)$ will be of the form $(\UU_\bullet\oplus \UU_\bullet^\vee, \gamma_{\mathbf A} \oplus 0)$, where $\mathbf A =(A_1,\ldots, A_{s-2})$ is a tuple of \emph{antisymmetric} matrices of size $p$. We are thus left with the problem of finding a semistable point in $E(p,p,s-2)$ consisting of antisymmetric matrices.

\smallskip
This is slightly more elaborate than the case of symmetric matrices. If $p$ is even, any choice of antisymmetric matrices $A_1,\ldots, A_s$ with $A_1$ invertible will give a semistable point. If $p$ is odd, on the other side, every antisymmetric matrix has a non-trivial kernel. In particular, there is no stable point in $E(p,p,2)$ consisting of antisymmetric matrices by Lemma \ref{l:Semistability2Matrices}, so one needs at least $3$ matrices. 

Let us assume $p$ is odd and write $p = 2p'+1$. Let $(e_0,\ldots e_{2p'})$ denote the canonical basis of $\C^p$. Define $A_1, A_2, A_3$ by
\begin{itemize}
\item $A_1(e_0) = 0$ and for $1\leq k \leq p'$, $A_1(e_{2k-1}) = e_{2k}$ and $A_1(e_{2k}) = - e_{2k-1}$,
\item $A_2(e_0) = 0$ and for $1\leq k \leq p'$, $A_1(e_{2k-1}) = k e_{2k}$ and $A_1(e_{2k}) = - k e_{2k-1}$,
\item $A_3(e_0) = - \sum_{k=1}^{p'} e_{2k}$ and for $1\leq k \leq p'$, $A_1(e_{2k-1}) = A'_3(e_{2k-1})$ and $A_1(e_{2k}) = A'_3(e_{2k}) + e_0$,
where $A'_3$ is an invertible antisymmetric endomorphism of $\mathrm{Span}(e_1,\ldots, e_{2p'})$ such that, for every subset $I$ of $\{1,\ldots, p'\}$, the image of $\mathrm{Span}(e_{2i-1})_{i\in I}$ is not contained in $\mathrm{Span}(e_{2i})_{i\in I}$. Note that the set of all such $A'_3$ is the complement of a finite union of proper algebraic subsets of the space of antisymmetric matrices of size $2p'$. In particular, it is non-empty.
\end{itemize}
We leave to the reader the verification that $A_1$, $A_2$ and $A_3$ are antisymmetric. 

We claim that $(A_1,A_2,A_3)$ is a semistable point of $E(p,p,3)$. By contradiction, assume that there exists $V$ and $W \subset \C^p$ such that $A_i(V) \subset W$ for $i = 1,2,3$ and $\dim W < \dim V$. Since the kernel of $A_1$ is $\C e_0$ we have that $V$ contains $e_0$ and $\dim W = \dim V -1$. We can thus write $V = \C e_0 \oplus V'$ where $V' \subset e_0^\perp$. Since $e_0$ is also the kernel of $A_2$, we have $W = A_1(V') = A_2(V')$.

Remark that $A_1$ and $A_2$ preserve $e_0^\perp$. Let $A'_1$ and $A'_2$ be the induced endomorphisms of $e_0^\perp$. Then $V'$ is invariant by ${A'_1}^{-1}A_2$ and it is thus spanned by eigenvectors of ${A'_1}^{-1}A_2$. Moreover, we also have $A_3(V') \subset W \subset e_0^\perp$. Since $\langle A_3(e_{2k}), e_0\rangle = 1$, we must have $V' = \mathrm{Span}(e_{2i-1})_{i\in I}$ for some subset $I$ of $\{1,\ldots, p'\}$. But we then have $A_3(V') \subset A_1(V') = \mathrm{Span}(e_{2i})_{i\in I}$, which contradicts the condition on $A_3$.

Therefore $(A_1,A_2,A_3)$ is a semistable point of $E(p,p,3)$ and every choice of antisymmetric matrices $(A_i)_{4\leq i \leq s-2}$ will give a semistable point in $E(p,p,s-2)$ corresponding to a semistable $\SO^*(2p)$ point in $\MM(\alpha,\beta,p)$.
\end{proof}

Similarly to the $\Sp(2p,\R)$ case, one could parametrize the coresponding compact $\SO^*(2p)$ relative components by a GIT quotient of the space of tuples of antisymmetric matrices of size $p$ by the action of $\GL(p,\C)$.

\subsubsection{Other Lie groups of Hermitian type} \label{sss:SO(2,n)}

We expect compact components in relative character varieties of punctured spheres to exist for all Hermitian Lie groups. The simple Hermitian Lie groups that are not covered by the present paper are the $\SO(2,p), p\geq 5$ and some real forms of $\mathrm E_6$ and $\mathrm E_7$.\footnote{The case $p\leq 4$ are covered, thanks to the following exceptional isogenies: $\SO(2,1)\simeq \SL(2,\R)$, $\SO(2,2)\simeq \SL(2,\R)\times \SL(2,\R)$, $\SO(2,3)\simeq \Sp(4,\R)$ and $\SO(2,4) \simeq \SU(2,2)$.}

Note that the standard embedding $\iota: \SO(2,p)\hookrightarrow \SU(2,p)$ induces a totally real embedding of symmetric spaces. Hence any representation into $\rho: \Gamma_{0,s} \to \SO(2,p)$ composed with $\iota$  has Toledo invariant $0$, and there cannot be a non-constant $\iota\circ \rho$-equivariant holomorphic or anti-holomorphic map from $\tilde \Sigma_{0,s}$ to the symmetric space of $\SU(2,p)$. Thus our domain $\Omega$ of Section \ref{s:PropertiesOfOurReps} in the $\SU(2,p)$ character variety does not contain any representation into $\SO(2,p)$.

In order to embed holomorphically the symmetric space of $\SO(2,p)$ into that of some $\SU(p',p')$, one should consider the spin representation
\[\Spin(2,p)\to \SU\left(2^{\lfloor \frac{p-1}{2}\rfloor}, 2^{\lfloor \frac{p-1}{2}\rfloor}\right)~.\]
A careful analysis of the spin representation might show that the domain $\Omega\subset \Rep(\Gamma_{0,s},  \SU(2^{\lfloor \frac{p-1}{2}\rfloor}, 2^{\lfloor \frac{p-1}{2}\rfloor}))$ constructed in Theorem \ref{t:ExistenceCompactComponents} contains representations that factor through $\Spin(2,p)$.

Alternatively, one could look directly at parabolic $\SO(2,p)$-Higgs bundles and carry the same analysis as in Section \ref{s:mainsection} in this new context (\emph{i.e.} find a small range of weights and Toledo invariant which force the Higgs field to be nilpotent). Though some computations show that this approach is promising, we felt that it would make the paper unnecessarily long.

\subsection{Restriction to a subsurface} \label{ss:RestrictionSubsurface}

So far, we only proved Theorems \ref{t:ExistenceCompactComponents} and~\ref{t-thm2} in the following cases:
\begin{itemize}
\item $G= \SU(p,q)$ and $s > \frac{p}{q}+ \frac{q}{p}+2$,
\item $G= \Sp(2p,\R)$ and $s\geq 5$ odd,
\item $G = \SO^*(2p)$ and $s\geq 5$ odd.
\end{itemize}

Here we deduce Theorems \ref{t:ExistenceCompactComponents} and \ref{t-thm2} for any value of $s\geq 3$, by arguing that one can always cut a sphere with a lot of punctures into spheres with fewer punctures and restrict representations in a compact component to these subsurfaces.

\smallskip
To be more precise, let $G$ be one of the Lie groups $\SU(p,q)$, $\Sp(2p,\R)$ or $\SO^*(2p)$. Let $s_0$ be greater than $\frac{p}{q}+\frac{q}{p}+2$ if $G=\SU(p,q)$ and $s_0$ be odd and greater than $4$ if $G=\Sp(2p,\R)$ or $\SO^*(2p)$. 

We will say that an open domain $\Omega$ of the character variety $\Rep(\Sigma_{0,s_0},G)$ is \emph{foliated by compact relative components} if for every representation $\rho$ in $\Omega$, the connected component of $\rho$ in its relative character variety is compact and contained in $\Omega$. We proved in Sections \ref{ss:OpenSetCompactCpts} and \ref{ss:MoreLieGroups} that such a domain exists. 

Now, let $b$ be an oriented simple closed curve that separates $\Sigma_{0,s_0}$ into a sphere with $s$ holes $\Sigma'$ and a sphere with $s_0-s+2$ holes $\Sigma''$. We choose a point on $b$ as a basepoint for our fundamental groups, so as to identify $\pi_1(\Sigma')$ and $\pi_1(\Sigma'')$ with subgroups of $\pi_1(\Sigma)$.

We have the following presentations:
\[\Gamma_{0,s_0} = \langle c_1, \ldots, c_{s_0} \mid c_1 \ldots c_{s_0} = 1\rangle~,\]
\[\Gamma_{0,s} \cong \pi_1(\Sigma') = \langle c_1, \ldots , c_{s-1}, b \mid c_1 \ldots c_{s-1} b = 1\rangle~,\]
\[\Gamma_{0,s_0-s+2} \cong \pi_1(\Sigma'') = \langle c_s, \ldots, c_{s_0}, b \mid b^{-1} c_s \ldots c_{s_0} = 1\rangle~.\\ \]

There is a well-defined restriction map
\[\mathrm{Res}:\left\{\begin{array}{rcl} \Rep(\Sigma_{0,s_0},G) & \longrightarrow & \Rep(\Sigma_{0,s},G)\\
 \left[ \rho\right] & \mapsto & \left[ \rho_{\vert \pi_1(\Sigma')}\right]
\end{array}\right. ,
\]
which is clearly open.

Let $\Omega$ be an open domain in $\Rep(\Sigma_{0,s_0},G)$ foliated by compact relative components. Let $\Omega'$ be the subset of $\Omega$ consisting of representations $\rho$ such that $\rho(b)$ is diagonalizable with distinct eigenvalues. Note that $\Omega'$ is open and non-empty since the map $\rho \mapsto \rho(b)$ is open. Theorem \ref{t:ExistenceCompactComponents} in full generality now follows from the following proposition:

\begin{prop}
The domain $\mathrm{Res}(\Omega') \subset \Rep(\Sigma_{0,s},G)$ is foliated by compact connected components.
\end{prop}

\begin{proof}
Let $\rho_0$ be a representation in $\Omega'$. Denote respectively by $\rho_0'$ and $\rho_0''$ the restrictions of $\rho_0$ to $\pi_1(\Sigma')$ and $\pi_1(\Sigma'')$. Define
\[h= (\rho_0(c_1), \ldots, \rho_0(c_{s_0}))~,\]
\[h' = (\rho_0(c_1),\ldots, \rho_0(c_{s-1}), \rho(b))\]
and
\[h'' = (\rho(b)^{-1}, \rho_0(c_s), \ldots, \rho_0(c_{s_0}))~.\]
We want to prove that the connected component of $\rho_0'$ in $\Rep_{h'}(\Sigma_{0,s},G)$ is compact and contained in $\mathrm{Res}(\Omega')$.

Let $K$ denote the subset of $\Rep_{h}(\Gamma_{0,s},\SU(p,q)) \cap \Omega$ consisting of representations $\rho$ such that $\rho(b)$ is conjugate to $\rho_0(b)$.

Since $\rho_0(b)$ is diagonalizable, its conjugation orbit is closed. The set $K$ is therefore closed in $\Rep_{h}(\Gamma_{0,s},\SU(p,q)) \cap \Omega$, hence compact. It is moreover contained in $\Omega'$.

By definition of $K$, the restriction map $\mathrm{Res}$ sends $K$ to $\to \Rep_{h'}(\Gamma_{0,s},G)$. Let us prove that its image contains the connected component of $[\rho_0']$. Let $[\rho_1']$ be a point in this connected component and let $[\rho_t']_{0\leq t\leq 1}$ be a continuous path from $[\rho_0']$ to $[\rho_1']$. For each $t$, we know that $\rho_t'(b)$ belongs to $C(\rho_0(b))$. Since $\rho_0(b)$ has distinct eigenvalues, this implies that $\rho_t'(b)$ and $\rho_0(b)$ are conjugate. Let $g_t\in G$ be such that
\[g_t\rho'(b) g_t^{-1} = \rho_0(b)~.\]
We can assume that $g_t$ varies continuously with $t$.
Define
\[\rho_t: \Gamma_{0,s_0} \to \SU(p,q)\]
such that
\[{\rho_t}_{\mid \Gamma'} = g_t\rho' g_t^{-1}\]
and
\[{\rho_t}_{\mid \Gamma''} = \rho_0''~.\]
By construction, we have $[\rho_t]\in \Rep_h(\Sigma_{0,s_0},G)$ for all $t$. Since $\Omega$ contains the connected component of $\rho_0$ in its relative character variety, we deduce that $\rho_t\in \Omega$ for all $t$. Moreover, $\rho_t(b)$ is conjugate to $\rho_0(b)$. Hence $[\rho_t]$ belongs to $K$. Finally,
${\rho_t}_{\vert\Sigma'}$ is conjugate to $\rho_t'$. Thus $[\rho_t']$ belongs to $\mathrm{Res}(K)\subset \mathrm{Res}(\Omega')$.
\end{proof}

Finally, to prove Theorem \ref{t-thm2} in full generality, we prove the following:
\begin{prop}
Every representation $\rho \in \mathrm{Res}(\Omega)$ satisfies the  properties of Theorem \ref{t:PropertiesRepresentationsPrecise}.
\end{prop}

\begin{proof}
Properties $(ii)$ and $(iii)$ immediately pass to the restriction of a representation to a subsurface. Property $(i)$ is slightly more subtle.

A representation $j: \Gamma_{0,s} \to \Isom^+(\H^2)$ is called \emph{Fuchsian} if it is discrete and faithful and $j(\Gamma_{0,s})\backslash \H^2$ is biholomorphic to a punctured Riemann sphere.

For any Fuchsian representation $j$ of $\pi_1(\Sigma') \cong \Gamma_{0,s}$, one can find a sequence of Fuchsian representations $(j_n)_{n\in \N}$ of $\Gamma_{0,s_0}$ such that $({j_n}_{\vert\Sigma'})_{n\in\N}$ converges to $j$. 

Now let $\rho$ be a representation in $\Omega$. By Theorem \ref{t:PropertiesRepresentationsPrecise}, for every $n\in \N$, there exists a $(j_n, \rho)$-equivariant holomorphic map $f_n$ from $\H^2$ to the symmetric space $X$ of $G$. All this maps are $1$-Lipschitz with respect to the Kobayashi metric of $X$. After extracting a subsequence, $(f_n)_{n\in\N}$ converges to a $(j,\rho_{\vert \Sigma'})$-equivariant holomorphic map from $\H^2$ to $X$. Since $j$ can be any Fuchsian representation, this shows that $\rho_{\vert \Sigma'}$ satisfies Property $(i)$.

\end{proof}

%%%%%%%%%%%%%%%%%%%%%
% References
%%%%%%%%%%%%%%%%%%%%%

\newcommand{\etalchar}[1]{$^{#1}$}
\ifx\undefined\bysame
\newcommand{\bysame}{\leavevmode\hbox to3em{\hrulefill}\,}
\fi

\end{document}